\documentclass[12pt]{amsart}
\usepackage[utf8]{inputenc}
\usepackage[T1]{fontenc}
\usepackage{amsfonts}
\usepackage[leqno]{amsmath}
\usepackage{accents}
\usepackage{amssymb, comment, float}
\usepackage[dvipsnames]{xcolor}
\usepackage[foot]{amsaddr}
\usepackage[shortlabels]{enumitem}
\setlist[itemize]{leftmargin=11mm}
\setlist[enumerate]{leftmargin=11mm}
\usepackage{mathdots}
\usepackage[a4paper, footskip=0.5in,  headheight = 0.5in, top=1.25in, bottom=1.25in,  right=1in,  left=1in]{geometry}
\usepackage{hyperref}
\hypersetup{
colorlinks,
linkcolor=black,
urlcolor=Blue,
citecolor=black,}
\usepackage{cleveref}
\usepackage[center]{caption}
\usepackage{eufrak}
\usepackage{marvosym}     
\usepackage{latexsym}
\usepackage[nodayofweek]{datetime}
\usepackage{tikz}
\usepackage{subfig}
\usepackage{thm-restate}
\usepackage{verbatim}

\newtheorem{statement}{statement}[section]
\newtheorem{theorem}[statement]{Theorem}
\newtheorem{lemma}[statement]{Lemma}
\newtheorem{conjecture}[statement]{Conjecture}
\newtheorem{corollary}[statement]{Corollary}

\newtheorem{observation}[statement]{Observation}

\DeclareMathOperator{\tw}{tw}

\DeclareMathOperator{\tri}{\triangleleft}
\DeclareMathOperator{\cone}{{\text{\sf{cone}}}}

\def\dd{\hbox{-}}   

\newcommand{\poi}{\mathbb{N}} 

\newcounter{tbox}
\newcommand{\sta}[1]{\medskip\medskip\refstepcounter{tbox}\noindent{\parbox{\textwidth}{(\thetbox) \emph{#1}}}\vspace*{0.3cm}}
\newcommand{\mylongtitle}[1]{%
  \ifodd\value{page}%
    \protect\parbox{0.97\linewidth}{#1}\hfill%
  \else%
    \hfill\protect\parbox{0.97\linewidth}{#1}%
  \fi%
}

\title[Induced subgraphs and tree decompositions XI.]{Induced subgraphs and tree decompositions\\
XI. Local structure in even-hole-free graphs of large treewidth}

\author{Bogdan Alecu$^{\ast \ast \mathparagraph}$}
\author{Maria Chudnovsky$^{\ast \amalg}$}
\author{Sepehr Hajebi $^{\mathsection}$}
\author{Sophie Spirkl$^{\mathsection \parallel}$}

\thanks{$^{\ast}$ Princeton University, Princeton, NJ, USA}
\thanks{$^{**}$ School of Computing, University of Leeds, Leeds, UK}
\thanks{$^{\mathsection}$ Department of Combinatorics and Optimization, University of Waterloo, ON, CA}
\thanks{$^{\amalg}$ Supported by NSF-EPSRC Grant DMS-2120644 and by AFOSR grant FA9550-22-1-0083.} 
     \thanks{$^{\mathparagraph}$ Supported by DMS-EPSRC Grant EP/V002813/1.} 
\thanks{$^{\parallel}$ We acknowledge the support of the Natural Sciences and Engineering Research Council of Canada (NSERC), [funding reference number RGPIN-2020-03912].
Cette recherche a \'et\'e financ\'ee par le Conseil de recherches en sciences naturelles et en g\'enie du Canada (CRSNG), [num\'ero de r\'ef\'erence RGPIN-2020-03912]. This project was funded in part by the Government of Ontario.}
\date {\today}

\begin{document}
\maketitle

\begin{abstract}
We prove a conjecture of Sintiari and Trotignon that every even-hole-free graph of sufficiently large treewidth contains a four-vertex induced subgraph with at least five edges (that is, either the four-vertex complete graph or the unique four-vertex graph with five edges, also known as the \textit{diamond}). 

In fact, we prove two stronger results: (a) For every $K_4$-free chordal graph $H$, every even-hole-free graph of sufficiently large treewidth contains either a four-vertex complete subgraph or an induced subgraph isomorphic to $H$ (when $H$ is the diamond, this yields their conjecture); and (b) For every $K_3$-free chordal graph $H$ (equivalently, for every forest $H$) and every $t \in \mathbb{N}$, every even-hole-free graph of sufficiently large treewidth contains either a $t$-vertex complete subgraph or an induced subgraph obtained from $H$ by adding a universal vertex (when $t=4$ and $H$ is the three-vertex path, this yields their conjecture).

The choice of $H$ in both result is best possible: (a) fails for every graph $H$ that is not $K_4$-free and chordal, and (b) fails for every graph $H$ that is not a forest.
\end{abstract}

\section{Introduction}\label{sec:intro}

Graphs in this paper have finite vertex sets, no loops and no parallel edges. Let $G=(V(G),E(G))$ be a graph. For $X\subseteq V(G)$, we use both $X$ and $G[X]$ to denote the subgraph of $G$ induced on $X$, and we write $G\setminus X$ for the graph obtained from $G$ by removing the vertices in $X$. We say that $G$ \textit{contains} a graph $H$ if $H$ is isomorphic to an induced subgraph of $G$; otherwise, we say $G$ is \textit{$H$-free}. The \textit{treewidth} of a graph $G$, denoted $\tw(G)$, is the smallest $w\in \poi\cup \{0\}$ for which there is a tree $T$ and a non-null subtree $T_v$ of $T$ for each vertex $v$ of $G$, such that the subtrees corresponding to adjacent vertices intersect, and each vertex of $T$ belongs to at most $w+1$ of the subtrees.

What induced subgraphs should be excluded from a graph in order to bound its treewidth? Motivated by the wide range of structural and algorithmic applications of treewidth, this question has received enormous attention in recent years. The grid theorem of Robertson and Seymour, \Cref{wallminor} below, answers the analogous question for minors and subgraphs. 

\begin{theorem}[Robertson and Seymour \cite{GMV}]\label{wallminor}
For every $t\in \poi$, there exists $w\in \poi$ such that
every graph with no minor (or equivalently, no subgraph) isomorphic to any subdivision of $W_{t\times t}$ has treewidth at most~$w$.
\end{theorem}
(For every $t\in \poi$, we denote by $W_{t\times t}$ the $t$-by-$t$ hexagonal grid, also known as the \textit{$t$-by-$t$ wall}; see Figure~\ref{fig:4basic}. It is well known \cite{diestel} that every subdivision of $W_{t\times t}$ has treewidth $t$.)
\medskip

Excluding subdivided walls as induced subgraphs, however, is not sufficient to bound the treewidth. Complete graphs, complete bipartite graphs, and the line graphs of subdivided walls are all examples of graphs with arbitrarily large treewidth and no induced subgraph isomorphic to any subdivision of $W_{3\times 3}$ (recall that the {\em line graph} $L$ of a graph $F$ has vertex set $E(F)$, with two vertices of $L$ adjacent whenever the corresponding edges of $F$ share an endpoint). These three constructions are still quite simple; thus, we place them alongside subdivided walls themselves to obtain a list of ``basic'' obstructions to bounded treewidth. For $t\in \poi$, a \textit{$t$-basic obstruction} is a graph isomorphic to $K_{t+1}$, $K_{t,t}$, some subdivision of $W_{t\times t}$, or the line graph of some subdivision of $W_{t\times t}$ (see Figure~\ref{fig:4basic}).
\begin{figure}[t!]
    \centering
    \includegraphics[scale=0.7]{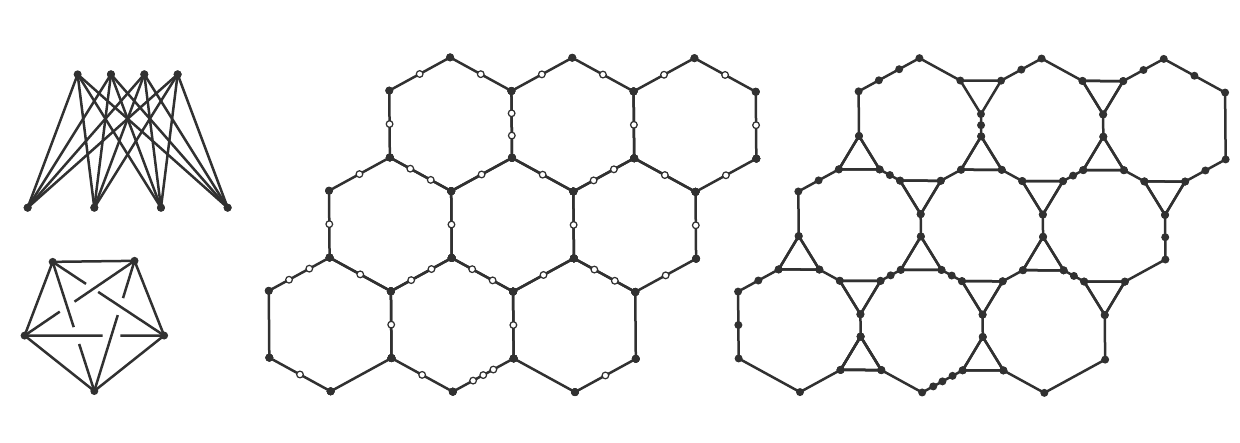}
    \caption{The $4$-basic obstructions, with a subdivided $4$-by-$4$ wall in the middle and its line graph on the right.}
    \label{fig:4basic}
\end{figure}

It is well known \cite{diestel} that the $t$-basic obstructions all have treewidth exactly $t$. Thus, a necessary condition for bounded treewidth is to exclude, for some $t\geq 1$, a $t$-basic obstruction of each type as an induced subgraph. If this were also sufficient, it would yield an elegant counterpart of \Cref{wallminor} for induced subgraphs; indeed, it is not hard to show (see \cite{twvii}) that excluding the $2$-basic obstructions does yield bounded treewidth. Unfortunately, the picture already breaks down here: the treewidth can be arbitrarily large in graphs that exclude the $3$-basic obstructions, and this remains true even if we exclude, along with $K_4$, a much simpler family of graphs than the other three basic obstructions: even holes.

A \textit{hole} in a graph $G$ is an induced cycle on at least four vertices; an \textit{even hole} is a hole of even length; and a graph is \textit{even-hole-free} if it contains no even hole. It is easy to observe that all $3$-basic obstructions except for $K_4$ contain even holes (see Figure~\ref{fig:basicwith3PC}, and also Section~\ref{sec:defns} for the definitions of a ``theta'' and a ``prism''). 

Therefore, being (even hole, $K_4$)-free is a much stronger condition than excluding the $3$-basic obstructions. Surprisingly, Sintiari and Trotignon showed that even under this stronger assumption, the treewidth can be arbitrarily large, and that excluding a finite number of \textit{odd} holes as well does not help either:

\begin{theorem}[Sintiari and Trotignon \cite{layered-wheels}]\label{thm:STLW}
    For all $h,w\in \poi$, there is an (even hole, $K_4$)-free graph $G_{h,w}$ with $\tw(G_{h,w})>w$ such that every hole in $G_{h,w}$ has length greater than $h$.
\end{theorem}

\begin{figure}[t!]
    \centering
    \includegraphics[scale=0.65]{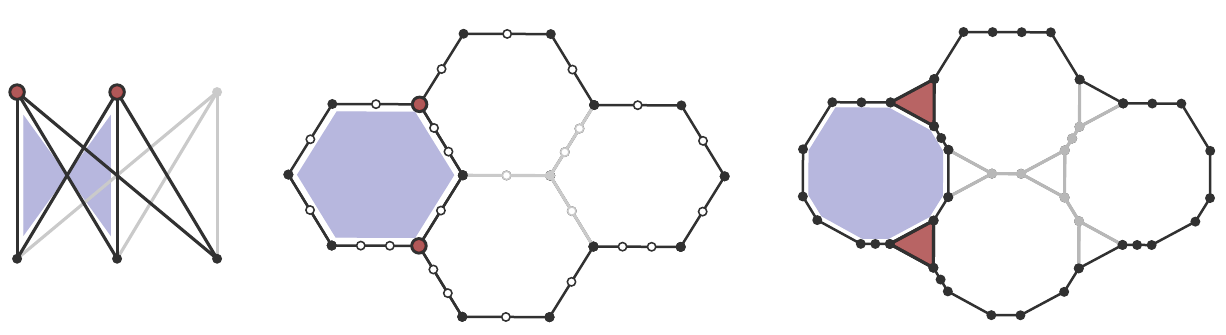}
    \caption{The $3$-basic obstructions other than the complete graph $K_4$, featuring a theta in $K_{3,3}$ (left), a theta in a subdivision of $W_{3\times 3}$ (middle), and a prism in the line graph of the same subdivided $W_{3\times 3}$ (right). An even hole in each theta and prism is highlighted. (See \Cref{sec:defns} for the definitions.)}
    \label{fig:basicwith3PC}
\end{figure}

The graphs $G_{h,w}$ in \Cref{thm:STLW}, explicitly constructed in \cite{layered-wheels}, are early examples of ``layered wheels,'' a rich class of constructions used \cite{newlayeredwheels, CT, layered-wheels} as counterexamples to several conjectures about the interplay between treewidth and induced subgraphs. 

Layered wheels are typically quite complicated (the even-hole-free one in \Cref{thm:STLW} remains the most complicated to date), and recent results suggest that, in a convincing sense, ``layered-wheel-like'' obstructions to bounded treewidth are the only remaining barrier to a full analog of \Cref{wallminor} for induced subgraphs (see \cite{barbados,tw16} for further discussion). 

As a result, there is considerable interest in understanding the structural properties of layered wheels, particularly the induced subgraphs that must appear in layered wheels of sufficiently large treewidth.  For instance, Sintiari and Trotignon observed \cite{layered-wheels} that, in order to maintain large treewidth while remaining even-hole-free, the graphs $G_{h,w}$ from \Cref{thm:STLW} tend strongly to contain induced ``diamonds'' (recall that the \textit{diamond} is the graph obtained from $K_4$ by deleting an edge). So they conjectured that diamonds are unavoidable induced subgraphs not only in the layered wheels of \Cref{thm:STLW}, but in \textit{all} (even hole, $K_4$)-free graphs of large treewidth.

\begin{conjecture}[Sintiari and Trotignon \cite{layered-wheels}]\label{diamondconj}
    Every (even hole, $K_4$)-free graph of large enough treewidth contains the diamond.
\end{conjecture}

In this paper, we prove \Cref{diamondconj}. It is worth noting, however, that the mere presence of induced diamonds may not seem like a particularly deep insight into the local structure of (even hole, $K_4$)-free graphs of large treewidth. After all, what other graphs $H$ are unavoidable induced subgraphs of such graphs? Our main result answers this question in full (and thereby proves \Cref{diamondconj} as a special case).

Let $H$ be a graph such that every (even hole, $K_4$)-free graph of sufficiently large treewidth has an induced subgraph isomorphic to $H$. In particular, $H$ must appear as an induced subgraph of the graphs $G_{h,w}$ in \Cref{thm:STLW} for all $h\in \poi$ and all sufficiently large $w\in \poi$. Therefore, $H$ must be $K_4$-free and contain no hole. Graphs with no holes are called \textit{chordal} (see Figure~\ref{fig:K4freechordal}).

\begin{observation}\label{obs:LWlocal}
    Let $H$ be a graph. Then every (even hole, $K_4$)-free graph of sufficiently large treewidth has an induced subgraph isomorphic to $H$ only if $H$ is $K_4$-free and chordal.
\end{observation}

Our main result is the converse to \Cref{obs:LWlocal}:

\begin{theorem}\label{mainchordal}
Every (even hole, $K_4$)-free graph of sufficiently large treewidth contains an induced subgraph isomorphic to $H$ if and only if $H$ is $K_4$-free and chordal.
\end{theorem}

\begin{figure}
    \centering
    \includegraphics[scale=0.6]{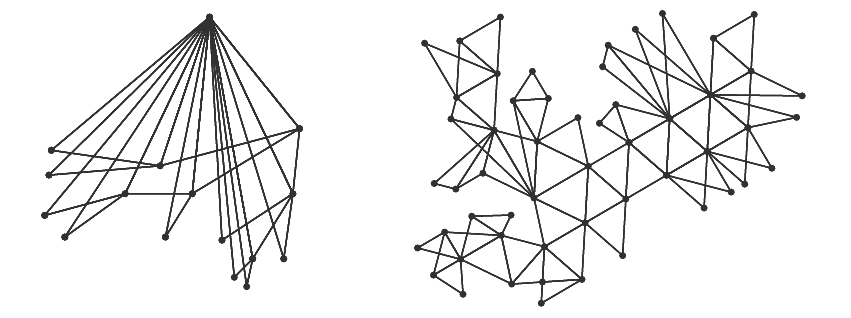}
    \caption{A coned tree as in Theorems~\ref{thm:mainforest+} (left) and \ref{thm:tree+} (middle), and an arbitrary $2$-tree (right).}
    \label{fig:K4freechordal}
\end{figure}

Since the diamond is $K_4$-free and chordal, \Cref{mainchordal} implies \Cref{diamondconj} immediately. But it also suggests something stronger: despite their intricate construction, even-hole-free layered wheels are canonical from a local perspective, in that every (even hole, $K_4$)-free graph of large treewidth, when ``zoomed in,'' displays a layered-wheel-like structure. 

\Cref{diamondconj} may also be interpreted in another way. Note that a diamond can be obtained from a two-edge path by adding a universal vertex. Thus, \Cref{diamondconj} states that every (even hole, $K_4$)-free graph of sufficiently large treewidth has an induced subgraph obtained from a two-edge path by adding a universal vertex. Our approach to proving \Cref{mainchordal} yields another extension of \Cref{diamondconj} in this interpretation, where ``$K_4$'' is replaced by ``$K_t$'' for arbitrary $t\in \poi$, and the ``two-edge path'' is replaced by an arbitrary forest (recall that a \textit{forest} is a graph with no cycles, that is, a graph whose components are trees). Surprisingly, by \Cref{thm:STLW}, this also characterizes forests. For a graph $F$, let $\cone(F)$ denote the graph obtained from $F$ by adding a vertex adjacent to all vertices in $V(F)$. Then $\cone(F)$ is $K_4$-free and chordal if and only if $F$ is a forest. We prove:

\begin{theorem}\label{thm:mainforest+}
For all $t\in \poi$, every (even hole, $K_t$)-free graph of sufficiently large treewidth has an induced subgraph isomorphic to $\cone(F)$, if and only if $F$ is a forest.
\end{theorem}

\Cref{thm:mainforest+} in particular shows that Conjecture~\ref{diamondconj} holds in general for even-hole-free graphs of bounded clique number, which in turn extends the main result of \cite{twiv}:

\begin{corollary}\label{diamondthm}
    For all $t\in \poi$, (even hole, diamond, $K_t$)-free graphs have bounded treewidth.
\end{corollary}

\subsection{Reduction to $2$-trees}\label{subs:ktree} Note that the ``only if'' implications in both Theorems~\ref{mainchordal} and \ref{thm:mainforest+} follow directly from Observation~\ref{obs:LWlocal}. For the ``if'' implication, we prove the corresponding results in the slightly more general class $\mathcal{E}$ of ($C_4$, theta, prism, even wheel)-free graphs (where $C_4$ denotes the $4$-vertex cycle, and a definition of ``even wheels'' will appear in the next section). At the same time, as a convenient technical step in our proof, we will reduce the ``if'' implication in Theorems~\ref{mainchordal} and \ref{thm:mainforest+} to the case where the excluded chordal graph is ``maximal'' with respect to its clique number.

Let us elaborate. For an integer $n$, we write $[n]$ for the set of all positive integers less than or equal to $n$ (so $[n]=\emptyset$ if $n\leq 0$). A well-known characterization of chordal graphs due to Dirac \cite{dirac} shows that for all integers $h,k\geq 1$, an $h$-vertex graph $H$ is a $K_{k+2}$-free chordal graph if and only if there exists a bijection $\pi:V(H)\rightarrow [h]$ such that for every $i\in [h-1]$, the neighborhood of $\pi(i)$ in $V(H)\setminus \pi([i])$ is a clique of cardinality at most $k$. This inspires the following definition: a \textit{$k$-tree} is a graph $\nabla$ which is either a $k$-vertex complete graph, or we have $|V(\nabla)|=h>k$ and there exists a bijection $\varpi_{\nabla}:V(\nabla)\rightarrow [h]$ such that for every $i\in [h-k]$, the set of neighbors of $\varpi_{\nabla}(i)$ in $V(\nabla)\setminus \varpi_{\nabla}([i])$, which we refer to as the \textit{forward neighbors of $\varpi_{\nabla}(i)$ in $\nabla$}, is a clique of cardinality \textit{exactly} $k$ in $\nabla$. For instance, $1$-trees are exactly trees. It follows that every $k$-tree is a connected, $K_{k+2}$-free chordal graph. More importantly, a partial converse holds, too (the proof is straightforward, yet we include it to keep the paper self-contained).
\begin{figure}
    \centering
    \includegraphics[scale=0.7]{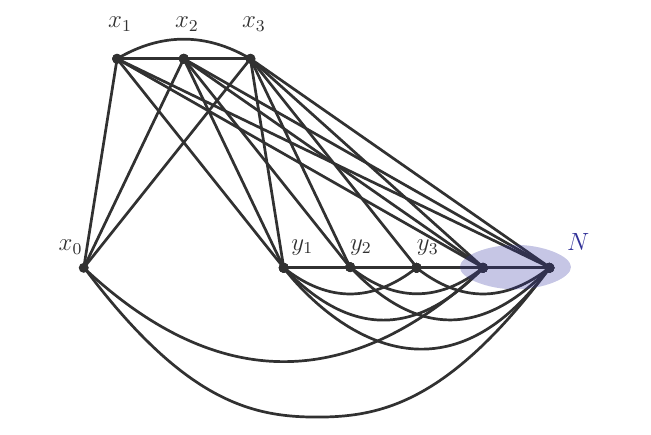}
\caption{Proof of Theorem~\ref{thm:k-tree} (when $k=5$ and $|N|=2$).}
    \label{fig:k-tree}
\end{figure}
\begin{theorem}\label{thm:k-tree}
    For every $k\geq 1$ and every $K_{k+2}$-free chordal graph $H$, there exists a $k$-tree $\nabla$ such that $H$ is an induced subgraph of $\nabla$.
\end{theorem}
\begin{proof}
   Note that every $K_{k+2}$-free chordal graph $H$ is an induced subgraph of a connected $K_{k+2}$-free chordal graph (which, for instance, may be obtained from $H$ by adding a new vertex with exactly one neighbor in each component of $H$). Therefore, we only need to show that for every connected $K_{k+2}$-free chordal graph $H$, there exists a $k$-tree $\nabla$ such that $H$ is an induced subgraph of $\nabla$. We prove this by induction on $|V(H)|=h$. The case $h=1$ is trivial, so assume that $h>1$. 
   
   Let $\pi:V(H)\rightarrow [h]$ be the bijection obtained from Dirac's result. Thus, for every $i\in [h-1]$, the set of neighbors of $\pi(i)$ in $V(H)\setminus \pi([i])$ is a clique on at most $k$ vertices in $H$. Let $\pi(1)=x_0$, let $H^-=H\setminus \{x_0\}$, and let $N$ be the set of neighbors $x_0$ in $H$. Then $N$ is a non-empty clique on at most $k$ vertices in $H$, which in turn implies that $H^-$ is a connected $K_{k+2}$-free chordal graph on $h-1$ vertices and $N$ is a non-empty clique on at most $k$ vertices in $H^-$.
    By the induction hypothesis, there is a $k$-tree $\nabla^-$ such that $H^-$ is an induced subgraph of $\nabla^-$. In particular,  $N$ is a non-empty clique on at most $k$ vertices in  $\nabla^-$. We deduce:
   
   \sta{\label{st:maximalclique}There is a clique $K$ of cardinality $k$ in $\nabla^-$ such  that $N\subseteq K$.}

   This is immediate if $\nabla^-$ is a complete graph. So we may assume that $\nabla^-$ is a $k$-tree that is not complete. Choose $x\in N$ with $\varpi_{\nabla^-}^{-1}(x)\in [h-1]$ as small as possible. Let $M$ be the set of forward neighbors of $x$ in $\nabla^-$ . Then $M\cup \{x\}$ is a clique of cardinality $k+1$ in $\nabla^-$ which contains the clique $N$ of cardinality at most $k$. This proves \eqref{st:maximalclique}.

   \medskip

   Let $K$ be as in \eqref{st:maximalclique}. Fix an enumeration $\{y_i:i\in [k-|N|]\}$ of the elements of $K\setminus N$. We define $\nabla$ as follows. Let
$$V(\nabla)=V(\nabla^-)\cup \{x_i:i\in \{0,\ldots, k-|N|\}\};$$
      $$E(\nabla)=E(\nabla^-)\cup \left(\bigcup_{i=0}^{k-|N|}(\{x_iy, x_iy_j, x_ix_{j'}:  y\in N, j\in [i], j'\in [k-|N|]\setminus [i]\})\right).$$
Let $\varpi_{\nabla}(x_i)=i+1$ for all $i\in \{0,\ldots, k-|N|\}$ and
let    $\varpi_{\nabla}(z)=\varpi_{\nabla^-}(v)+k-|N|+1$ for all $v\in V(\nabla^-)$ (see Figure~\ref{fig:k-tree}). One may check that $\nabla$ is $k$-tree and $H$ is an induced subgraph of $\nabla$; we omit the details.
\end{proof}

For every $t\geq 1$, we denote by $\mathcal{E}_t$ the class of all $K_t$-free graphs in $\mathcal{E}$. In view of Theorem~\ref{thm:k-tree}, Theorems~\ref{mainchordal} and \ref{thm:mainforest+} follow, respectively, from Theorems~\ref{thm:main2-tree} and \ref{thm:tree+} below. We will prove both results in the last section.

\begin{theorem}\label{thm:main2-tree} For every $2$-tree $\nabla$, there exists an integer $\Upsilon=\Upsilon(\nabla)\geq 1$ such that every graph $G\in \mathcal{E}_4$ with $\tw(G)>\Upsilon$ contains $\nabla$.
\end{theorem}
\begin{theorem}\label{thm:tree+}
    For every integer $t\geq 1$ and every tree $T$, there exists an integer $\Gamma=\Gamma(t,T)\geq 1$ such that every $G\in \mathcal{E}_t$ with $\tw(G)>\Gamma$ contains $\cone(T)$.
\end{theorem}

\subsection{Outline}\label{subs:outline} We conclude the introduction with a rough account of our main ideas. Due to the fact that adding universal vertices to a tree results in a $2$-tree, the proofs of Theorems~\ref{thm:main2-tree} and \ref{thm:tree+} follow a similar outline, which we give below.

Let $G\in \mathcal{E}_t$ be a graph of huge treewidth. We will show that $G$ contains a copy of a given $2$-tree $\nabla$ as a subgraph, not necessarily induced, but close enough so that the outcomes of Theorems~\ref{thm:main2-tree} and \ref{thm:tree+} would follow immediately from known results. Our method is to grow this copy of $\nabla$ in $G$ through an inductive process,  adding one vertex at a time with respect to the ordering imposed on $V(\nabla)$ by $\varpi_{\nabla}$, reversed (so $\varpi_{\nabla}(1)$ is the last vertex to be added).

In order to grow our $2$-tree, we use an auxiliary structure that we call a ``kaleidoscope,'' consisting of many holes sharing a three-vertex path and otherwise vertex-disjoint. We then define a notion of ``mirroring,'' whereby, roughly speaking, a set of vertices is ``$d$-mirrored'' by a kaleidoscope if every vertex in the set has at least $d$ neighbors in each of the kaleidoscope's holes, outside of the shared three-vertex path.

The main ingredients to our argument are then as follows:

\begin{itemize}
\item In Section~\ref{sec:gbu} (and in particular in Theorem~\ref{thm:difftodiff}), we show that given an integer $d$, if a vertex is $1$-mirrored by a suitably large kaleidoscope, then it is in fact $d$-mirrored by a ``sub-kaleidoscope'' of prescribed size. In particular, when $d \geq 3$, it is a wheel center for all of the holes in the sub-kaleidoscope.

\item To begin the growing process, we show in Theorem~\ref{thm:kaleidoscopeexists} that, assuming our graph has large treewidth (or more specifically, that it has two vertices connected by many disjoint paths), we can produce, using the previous bullet point and results from an earlier paper in the series, a clique of size two which is $3$-mirrored by a large kaleidoscope. In general, our induction hypothesis will be that we can find a large kaleidoscope $3$-mirroring the $2$-tree we have constructed so far.

\item Given two adjacent wheel centers with the same rim, Theorem~\ref{thm:evenwheeltheta} provides some constraints on their neighborhoods along the rim. This allows us to find common neighbors of adjacent vertices in our $2$-tree on each of the holes of the large kaleidoscope $3$-mirroring it. These common neighbours are candidates for extending our $2$-tree.

\item Theorem~\ref{thm:mainblurry} is the core of our induction step. In it, we show that, by choosing the candidate carefully, we are able to guarantee that it has three neighbors on many of the other holes of the kaleidoscope. In other words, we are able to find a large sub-kaleidoscope $3$-mirroring the extended $2$-tree, thus maintaining the property we need for the induction.

\item Finding that right candidate requires a novel idea, presented in Section~\ref{sec:<3}, which is totally different from the material developed in the earlier papers in this series: we show that taking specific minors of $G$ keeps us in the class $\mathcal{E}$. This enables us to essentially ``pretend'' certain edges in $G$ do not exist, which in turn allows us to use the machinery from an earlier paper in order to find the right candidate for extending the $2$-tree.
\end{itemize}

 In Section~\ref{sec:defns}, we cover the terminology and the results from earlier papers in this series to be used in subsequent sections. We complete the proofs of Theorems~\ref{thm:main2-tree} and \ref{thm:tree+} in Section~\ref{sec:end}. 

\section{Preliminaries}
\label{sec:defns}
In this section, we set up our notation and terminology. We also  mention a few results from the earlier papers in this series \cite{twviii,twvii}.

 Let $G = (V(G),E(G))$ be a graph and let $x\in V(G)$. We denote by $N_G(x)$ the set of all neighbors of $x$ in $G$, and write $N_G[x]=N_G(x)\cup \{x\}$. For an induced subgraph $H$ of $G$ (not necessarily containing $x$), we define $N_H(x)=N_G(x) \cap H$ and $N_H[x]=N_H(x)\cup \{x\}$. Also, for $X\subseteq G$, we denote by $N_G(X)$ the set of all vertices in $G\setminus X$ with at least one neighbor in $X$, and define $N_G[X]=N_G(X)\cup X$. Let $X,Y \subseteq G$ be disjoint. We say $X$ is \textit{complete} to $Y$ if every vertex in $X$ is adjacent to every vertex in $Y$ in $G$, and $X$ is \emph{anticomplete}
to $Y$ if there is no edge in $G$ with an end in $X$ and an end in $Y$.

A {\em path in $G$} is an induced subgraph of $G$ that is a path. If $P$ is a path in $G$, we write $P = p_1 \dd \cdots \dd p_k$ meaning $V(P) = \{p_1, \dots, p_k\}$ and $p_i$ is adjacent to $p_j$ if and only if $|i-j| = 1$. We call the vertices $p_1$ and $p_k$ the \emph{ends of $P$}, and say that $P$ is \emph{from $p_1$ to $p_k$}. The \emph{interior of $P$}, denoted by $P^*$, is the set $P \setminus \{p_1, p_k\}$. The \emph{length} of a path is its number of edges (so a path of length at most one has empty interior). Similarly, if $C$ is a cycle, we write $C = c_1 \dd \cdots \dd c_k\dd c_1$ to mean that $V(C) = \{c_1, \dots, c_k\}$ and $c_i$ is adjacent to $c_j$ if and only if $|i-j|\in \{1,k-1\}$. The \textit{length} of a cycle is its number of vertices (or edges).

A {\em theta} is a graph $\Theta$ consisting of two non-adjacent vertices $a, b$, called the \textit{ends of $\Theta$}, and three pairwise internally disjoint paths $P_1, P_2, P_3$ from $a$ to $b$ of length at least two, called the \textit{paths of $\Theta$}, such that $P_1^*, P_2^*, P_3^*$ are pairwise anticomplete to each other. For a graph $G$, by a \textit{theta in $G$} we mean an induced subgraph of $G$ which is a theta.

A {\em prism} is a graph $\Pi$ consisting of two disjoint triangles $\{a_1,a_2,a_3\} $ and $\{b_1,b_2,b_3\}$, called the \textit{triangles of $\Pi$}, and three pairwise disjoint paths $P_1,P_2,P_3$ called the \textit{paths of $\Pi$}, where $P_i$ has ends $a_i,b_i$
for each $i\in \{1,2,3\}$, and for distinct $i,j\in \{1,2,3\}$, $a_ia_j$ and $b_ib_j$ are the only edges between $P_i$ and $P_j$. For a graph $G$, by a \textit{prism in $G$} we mean an induced subgraph of $G$ which is a prism.

A {\em wheel} in a graph $G$ is a pair $(C,v)$ where $C$ is a hole in $G$ and $v\in G\setminus C$ is a vertex with at least three neighbors in $C$.  An \textit{even wheel} in $G$ is a wheel $(C,v)$ in $G$ where $v$ has an even number of neighbors in $H$. We say $G$ is \textit{even-wheel-free} if there is no even wheel in $G$. 

See Figure~\ref{fig:forbidden_isgs} for a depiction of a theta, a prism, a pyramid and an even wheel. It is straightforward to check that the class $\mathcal{E}$ of all ($C_4$, theta, prism, even wheel)-free graphs contains the class of all even-hole-free graphs.

\begin{figure}[t!]
\centering
\includegraphics[scale=0.8]{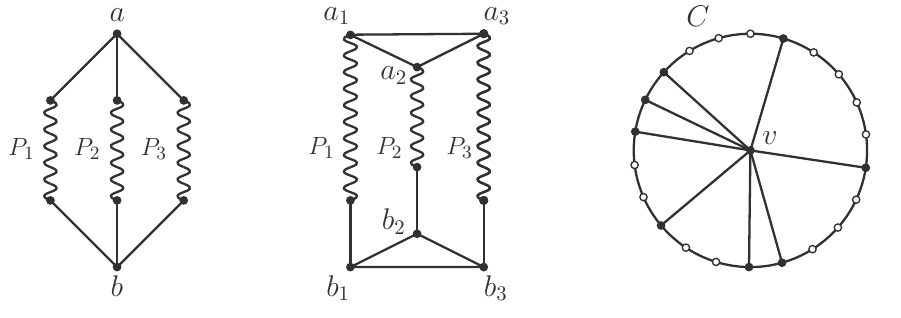}
\caption{From left to right, a theta, a prism and an even wheel. Squiggly lines represent paths of arbitrary length (possibly zero).}
\label{fig:forbidden_isgs}
\end{figure}

We now mention a few results from two previous papers in the current series \cite{twvii, twviii}. Let $k$ be a positive integer and let $G$ be a graph. A \textit{strong $k$-block} in $G$ is a set $B$ of at least $k$ vertices in $G$ such that for every $2$-subset $\{x,y\}$ of $B$, there exists a collection $\mathcal{P}_{\{x,y\}}$ of at least $k$ distinct and pairwise internally disjoint paths in $G$ from $x$ to $y$, where for every two distinct $2$-subsets $\{x,y\}, \{x',y'\}\subseteq B$ of $G$, and every choice of paths $P\in \mathcal{P}_{\{x,y\}}$ and $P'\in \mathcal{P}_{\{x',y'\}}$, we have $P\cap P'=\{x,y\}\cap \{x',y'\}$. In \cite{twvii},  it is proved that:

\begin{theorem}[Abrishami, Alecu, Chudnovsky, Hajebi and Spirkl \cite{twvii}]\label{noblocksmalltw_wall}
  For all integers $k,t\geq 1$, every graph with no  strong $k$-block and no induced subgraph that is a $t$-basic obstruction has bounded treewidth.
\end{theorem}
As discussed in the introduction, for every $t \geq 1$, a (theta, prism)-free graph has no induced subgraph that is a $t$-basic obstruction if and only if it is $K_t$-free. So the following is
immediate from Theorem~\ref{noblocksmalltw_wall}:

\begin{corollary}\label{cor:noblocksmalltw_Ct}
For all integers $k,t\geq 1$, there exists an integer $\beta=\beta(k,t)$ such that every (theta, prism, $K_t$)-free graph with no strong $k$-block has treewidth at most $\beta(k,t)$. \end{corollary}

Moreover, Theorem~\ref{banana} below reveals further information about the adjacency between different paths joining two vertices in a strong block. This was a major ingredient in our proof of Theorem~\ref{mainchordal} in \cite{twviii}.

\begin{theorem}[Abrishami, Alecu, Chudnovsky, Hajebi and Spirkl \cite{twviii}]\label{banana}
    For all integers $t,\nu\geq 1$, there exists an integer $\psi=\psi(t,\nu)\geq 1$ with the following property. Let $G$ be a (theta, prism, $K_t$)-free graph, let $a,b\in V(G)$ be distinct and non-adjacent and let $\mathcal{P}$ be a collection of pairwise internally disjoint paths in $G$ from $a$ to $b$ with $|\mathcal{P}|\geq \psi$. For each $P\in \mathcal{P}$, let $x_{P}$ be the neighbor of $a$ in $P$ (so $x_P\neq b$). Then there exist $P_1,\ldots, P_{\nu}\in \mathcal{P}$ such that:
    \begin{itemize}
    \item $\{x_{P_1},\ldots, x_{P_{\nu}},b\}$ is a stable set in $G$; and
    \item for all $i,j\in [\nu]$ with $i<j$, $x_{P_i}$ has a neighbor in $P_j^*\setminus \{x_{P_j}\}$.
    \end{itemize}
\end{theorem}

We also need a quantified version of Ramsey's classical theorem, which has appeared in several references; see, for instance, \cite{eta}.
\begin{theorem}[Ramsey \cite{multiramsey}, see also \cite{eta}]\label{classicalramsey}
For all integers $c,s\geq 1$, every graph $G$ on at least $c^s$ vertices contains either a clique of cardinality $c$ or a stable set of cardinality $s$. 
\end{theorem}

Finally, we include the following well-known lemma which follows directly from Theorem~\ref{classicalramsey} combined with Lemma 2 in \cite{lozin}.

\begin{lemma}\label{ramsey2}
For all integers $q,r,s,t\geq 1$ there exists an integer $o=o(q,r,s,t)\geq 1$ with the following property.  Let $G$ be a $(K_{s,s},K_t)$-free graph. Let $\mathcal{X}$ be a collection of pairwise disjoint subsets of $V(G)$, each of cardinality at most $r$, with $|\mathcal{X}|\geq o$. Then there are $q$ distinct sets $X_1,\ldots, X_q\in \mathcal{X}$ which are pairwise anticomplete in $G$.

\end{lemma}

\section{Contraptions: class-preserving minors in $\mathcal{E}$}\label{sec:<3}

 In this section, we develop the most critical step in our procedure for growing $2$-trees by iteratively obtaining common neighbors of prescribed pairs of adjacent vertices. The main result is Theorem~\ref{thm:<3}, which shows that, although $\mathcal{E}$ is a class closed under taking induced subgraphs, certain minors of certain graphs in $\mathcal{E}$ belong to $\mathcal{E}$. Recall that a \textit{minor} of $G$ is a graph that is obtained from $G$ by a sequence of vertex deletions, edge deletions and \textit{edge contractions}, where for an edge $xy$ of a graph $G$, the \textit{$xy$-contraction} of $G$ is the graph obtained from $G$ by identifying the two vertices $x$ and $y$ into a single vertex.    
 
 For a graph $G$ and two adjacent vertices $z_1,z_2\in V(G)$, we define the \textit{$z_1z_2$-contraption} of $G$, denoted $G\tri ^{z_1}_{z_2}$, to be the graph with the following specifications:
\begin{itemize}
    \item $V(G\tri_{z_2}^{z_1})=(V(G)\setminus \{z_1,z_2\})\cup \{z\}$;
    \item $G\tri_{z_2}^{z_1}[V(G)\setminus \{z_1,z_2\}]=G\setminus \{z_1,z_2\}$; and
    \item $N_{G\tri_{z_2}^{z_1}}(z)=N_G(z_1)\cap N_G(z_2)$.
\end{itemize}

See Figure~\ref{fig:c4needed}. In other words, the $z_1z_2$-contraption of $G$ is the minor of $G$ (without parallel edges) obtained by first contracting the edge $z_1z_2$ into a new vertex $z$, and then removing every edge in the resulting graph between $z$ and a vertex in the symmetric difference of $N_G(z_1)$ and $N_G(z_2)$. Our goal in this section is to prove the following:

\begin{theorem}\label{thm:<3}
    Let $G\in \mathcal{E}$ be a graph and let $z_1,z_2\in V(G)$ be distinct and adjacent such that $N_G(z_1)\cap N_G(z_2)$ is a stable set of vertices of degree at most three in $G$. Then we have  $G\tri_{z_2}^{z_1}\in \mathcal{E}$.
\end{theorem}

Two remarks:  first, as far as our application of Theorem~\ref{thm:<3} is concerned, it suffices to show that $G\tri_{z_2}^{z_1}$ is (theta, prism)-free (under the same hypotheses). This is thanks to Theorem~\ref{banana} holding true for the larger class of (theta, prism)-free graphs rather than just $\mathcal{E}$. Second, the proof of Theorem~\ref{thm:<3} (and so its application in the proof of Theorem~\ref{lem:mainblurryinduct}) is the only place in this paper where we use the assumption that $G$ is $C_4$-free. Nevertheless, as unfortunate as it may appear, excluding $C_4$ is necessary even if we ask for
 $G\tri_{z_2}^{z_1}$ not to ``be'' a theta; see Figure \ref{fig:c4needed}.

\begin{figure}
\centering
\includegraphics[scale=0.7]{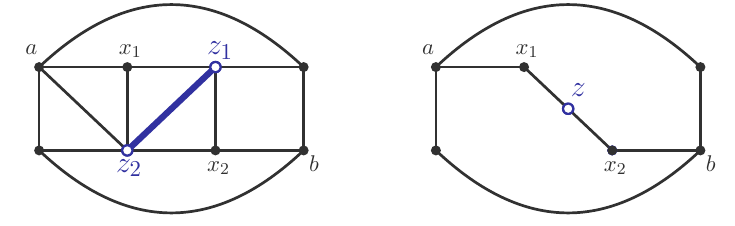}

\caption{Left: a (theta, prism, even wheel)-free graph $G$ containing a clique $\{z_1,z_2\}$ for which $N_G(z_1)\cap N_G(z_2)=\{x_1,x_2\}$ is a stable set of vertices of degree three in $G$ (observe that $G$ does contain $C_4$). Right: the graph $G\tri ^{z_1}_{z_2}$ which is a theta with ends $a,b$. }
\label{fig:c4needed}
\end{figure}

We now plunge into the proof of Theorem~\ref{thm:<3}, beginning with the following definition. Let $G$ be a graph, let $H$ be an induced subgraph of $G$ and let $v\in G\setminus H$. We say that:
\begin{enumerate}[(F1),leftmargin=15mm, rightmargin=7mm]
\item[(G)]\label{F1}  $v$ is \textit{$H$-good} if $|N_H(v)|=1$;
\item[(B)]\label{F2} $v$ is \textit{$H$-bad} if $N_H(v)$ is a clique in $H$ with at least two vertices; and
 \item[(U)]\label{F3} $v$ is \textit{$H$-ugly} if $N_H(v)$ is not a clique in $H$.
\end{enumerate}
So every vertex in $N_G(H)\subseteq G\setminus H$ is exactly one of $H$-good, $H$-bad, or $H$-ugly. The next result is an important ingredient for the proof of Theorem~\ref{thm:<3}. Similar results have also appeared in \cite{ksher, twiv}.

\begin{theorem}\label{thm:evenwheeltheta}
   Let $G$ be a (theta, prism, even wheel)-free graph, let $C$ be a hole in $G$ and let $z_1,z_2\in G\setminus C$ be distinct and adjacent, each with at least one neighbor in $C$. Assume that $z_1$ and $z_2$ have no common neighbor in $C$. Then
either both $z_1$ and $z_2$ are $C$-good and their (unique) neighbors in $C$ are distinct and adjacent, or exactly one of $z_1$ and $z_2$ is $C$-bad. Consequently, if $G\in \mathcal{E}$, then exactly one of $z_1$ and $z_2$ is $C$-bad.
\end{theorem}
\begin{proof}
 Note that if both $z_1$ and $z_2$ are $C$-bad, then since $z_1$ and $z_2$ have no common neighbor in $C$, it follows that $C\cup \{z_1,z_2\}$ is a prism in $G$, a contradiction. So we may assume without loss of generality that $z_1$ is either $C$-good or $C$-ugly. If $z_2$ is $C$-bad, then we are done. So we can consider the case that $z_2$ is also either $C$-good or $C$-ugly; in particular, since neither $(C,z_1)$ nor $(C,z_2)$ is an even wheel in $G$, it follows that for every $i\in \{1,2\}$, $|N_C(z_i)|$ is an odd integer. Assume first that both $z_1$ and $z_2$ are $C$-good, say $N_C(z_i)=\{x_i\}$ for $i\in \{1,2\}$. Then since $z_1$ and $z_2$ have no common neighbor in $C$, and $C\cup \{z_1,z_2\}$ is not a theta in $G$, it follows that $x_1$ and $x_2$ are distinct and adjacent in $G$, as required.
 
This leaves the case where one of $z_1$ and $z_2$, say the former, is $C$-ugly. Since $z_1$ and $z_2$ have no common neighbor in $C$, it follows that $N_C(z_2)\subseteq C\setminus N_C(z_1)$. Note that every component  of $C\setminus N_C(z_1)$ is a path in $C$ (and so in $G$). Moreover, for every component  $P$ of $C\setminus N_C(z_1)$, $C_P=N_C[P]\cup \{z_1\} $ is a hole in $G$. Since $C_P\cup \{z_2\}$ is not a theta in $G$, and $(C_P,z_2)$ is not even wheel in $G$, and $z_1$ and $z_2$ have no common neighbor in $C$, it follows that $z_2$ has an even number of neighbors in $P$. In conclusion, we have shown that $z_2$ has an even number of neighbors in each component of $C\setminus N_C(z_1)$. But then $z_2$ has an even number of neighbors in $C$, a contradiction. We conclude that either both $z_1$ and $z_2$ are $C$-good and their neighbors in $C$ are distinct and adjacent, or exactly one of $z_1$ and $z_2$ is $C$-bad. In addition, if $G\in \mathcal{E}$, then the first outcome does not hold, as otherwise $G[N_C[z_1]\cup N_C[z_2]]$ is isomorphic to $C_4$, a contradiction. This completes the proof of Theorem~\ref{thm:evenwheeltheta}.
\end{proof}

We also need the following lemma.
\begin{lemma}\label{lem:zintheta}
   Let $G\in \mathcal{E}$ be a graph and let $z_1,z_2\in V(G)$ be distinct and adjacent such that $N_G(z_1)\cap N_G(z_2)$ is a stable set of vertices of degree at most three in $G$. Let $z\in V(G\tri^{z_1}_{z_2})$ be as in the definition of $G\tri^{z_1}_{z_2}$ and let $W$ be an induced subgraph of $G\tri^{z_1}_{z_2}$ which is either a theta, or a prism, or an even wheel. Then there is a path $P$ in $W$ with ends $a,b$ for which the following hold.
    \begin{enumerate}[{\rm(a)}]
        \item\label{lem:zintheta_a} We have $z\in P\setminus (N_W[a]\cup N_W[b])$, and so $W\setminus P^*\subseteq G\setminus (N_G[z_1]\cap N_G[z_2])$.
        \item\label{lem:zintheta_b} The vertices in $P^*$ (including $z$) have degree two in $W$, and $a,b$ both have degree three in $W$. 
        \item\label{lem:zintheta_c} In the graph $G$, both $z_1$ and $z_2$ have a neighbor in $W\setminus P$.
    \end{enumerate}
\end{lemma}
\begin{proof}
    First, assume that there is no path $P$ in $W$ satisfying \ref{lem:zintheta}\ref{lem:zintheta_a} and \ref{lem:zintheta}\ref{lem:zintheta_b}. Since $G\in \mathcal{E}$, it follows that $z\in W$. Also, since $N_G(z_1)\cap N_G(z_2)$ is a stable set of vertices of degree at most three in $G$, it follows that $N_{G\tri^{z_1}_{z_2}}(z)$ is a stable set of vertices of degree at most two in $G\tri^{z_1}_{z_2}$. In particular, there is no wheel $(C,v)$ in $G\tri^{z_1}_{z_2}$ where $z\in N_C[v]$, and $z$ does not belong to a triangle of a prism in $G\tri^{z_1}_{z_2}$. Moreover, from the assumption that there is no path in $W$ satisfying \ref{lem:zintheta}\ref{lem:zintheta_a} and \ref{lem:zintheta}\ref{lem:zintheta_b}, it follows that there is no wheel $(C,v)$ in $G\tri^{z_1}_{z_2}$ where $z\in C\setminus N_C(v)$, and $z$ does not belong the interior of a path of a theta or a prism in $G\tri^{z_1}_{z_2}$. We deduce that $W$ is a theta in $G\tri^{z_1}_{z_2}$ and $z$ is an end of $W$. Let $z'\in V(G\tri^{z_1}_{z_2})\setminus N_{G\tri^{z_1}_{z_2}}[z]=V(G)\setminus (N_G[z_1]\cap N_G[z_2])$ be the other end of $W$ and let $P_1,P_2,P_3$ be the paths of $W$. Then for every $i\in [3]$, $P_i$ has ends $z,z'$, and for some $j\in \{1,2\}$, $z_j$ is not adjacent to $z'$ in $G$. On the other hand, for every $i\in [3]$, we have $N_{P_i}(z)\subseteq N_G(z_j)\cap (P_i\setminus \{z\})$. Thus, traversing $P_i\setminus \{z\}$ starting at $z'$, we may choose $x_i$ to be the first vertex in $N_G(z_j)\cap (P_i\setminus \{z\})$; it follows that $x_i\in P_i^*$. But then there is a theta in $G$ with ends $z_j,z'$ and paths $z_j\dd x_i\dd P_i\dd z'$ for $i\in [3]$, which violates $G\in \mathcal{E}$. This proves that there exists a path $P$ in $W$ with ends $a,b$ for which  \ref{lem:zintheta}\ref{lem:zintheta_a} and \ref{lem:zintheta}\ref{lem:zintheta_b} hold. 
    
    It remains to show that $P$ satisfies \ref{lem:zintheta}\ref{lem:zintheta_c}. Suppose for a contradiction, and without loss of generality,  that $z_1$ is anticomplete to $W\setminus P$ in $G$. Then $U=(P\setminus \{z\})\cup \{z_1\}$ is a  connected induced subgraph of $G$ with $a,b\in U$ such that $U\setminus \{a,b\}\subseteq P^*\cup \{z_1\}$ is anticomplete to $W\setminus P$ in $G$. Consequently, there exists a path $P_1$ in $U$ from $a$ to $b$ where $P^*_1$ is anticomplete to $W\setminus P$ in $G$. But now $(W\setminus P)\cup P_1$ is a theta, a prism or an even wheel in $G$ (depending on whether $W$ is a theta, a prism or an even wheel in $G$, respectively), which is impossible because $G\in \mathcal{E}$. This completes the proof of Lemma~\ref{lem:zintheta}.
\end{proof}
The next two lemmas, in turn, show that under the assumptions of Theorem~\ref{thm:<3}, the graph $G\tri^{z_1}_{z_2}$ is theta-free and prism-free.
\begin{lemma}\label{lem:<3notheta}
   Let $G\in \mathcal{E}$ be a graph and let $z_1,z_2\in V(G)$ be distinct and adjacent such that $N_G(z_1)\cap N_G(z_2)$ is a stable set of vertices of degree at most three in $G$. Then $G\tri^{z_1}_{z_2}$ is theta-free.
\end{lemma}

\begin{proof}
Suppose for a contradiction that there is a theta $W$ in $G\tri^{z_1}_{z_2}$. Let $z\in V(G\tri^{z_1}_{z_2})$ be as in the definition of $G\tri^{z_1}_{z_2}$. Let $P$ be the path in $W$ with ends $a,b$ satisfying Lemma~\ref{lem:zintheta}. It follows from Lemma~\ref{lem:zintheta}\ref{lem:zintheta_a} and \ref{lem:zintheta_b} that $a,b$ are the ends of $W$, $P$ is a path of $W$, and we have $z\in P\setminus (N_{P}[a]\cup N_{P}[b])$. Let $Q_1,Q_2$ be the paths of $W$ distinct from $P$; so $Q_1$ and $Q_2$ both have ends $a,b$, as well. Let $C=Q_1\cup Q_2$. Then $C$ is a hole in $G\setminus \{z_1,z_2\}$ and we have $C=W\setminus P^*$. 

From the definition of $G\tri^{z_1}_{z_2}$, it follows that $W\setminus \{z\}\subseteq G\setminus \{z_1,z_2\}$ and $\{z_1,z_2\}$ is complete to $N_{P}(z)$. As a result, for every $i\in \{1,2\}$, there are two paths $P_{a,i},P_{b,i}$ in $(P\setminus \{z\})\cup \{z_i\}$ from $a$ to $z_i$ and from $b$ to $z_i$, respectively, such that $P_{a,i}\setminus \{z_i\}$ and $P_{b,i}\setminus \{z_i\}$ are disjoint and anticomplete in $G$, and both $P_{a,i}^*$ and $P^*_{b,i}$ are disjoint from and anticomplete to $C\setminus \{a,b\}$ in $G$. We claim that:

\sta{\label{st:oneofa,b}Let $i\in \{1,2\}$. Then either $a$ is adjacent to $z_i$ in $G$, or $N_C(z_i)\subseteq N_{Q_j}[b]$ for some $j\in \{1,2\}$. Similarly, either $b$ is adjacent to $z_i$ in $G$, or $N_C(z_i)\subseteq N_{Q_j}[a]$ for some $j\in \{1,2\}$.}

We only need to show that for every $i\in \{1,2\}$, either $a$ is adjacent to $z_i$ in $G$, or $N_C(z_i)\subseteq N_{Q_j}[b]$ for some $j\in \{1,2\}$. Suppose not. Then we may assume without loss of generality, that $a$ is not adjacent to $z_1$ in $G$, and there is a vertex in $Q_1^*\setminus N_{Q_1}(b)$ which is adjacent to $z_1$ in $G$. It follows that $P_{a,1}$ has length at least two, and that there is a path $R$ of length at least two in $G$ from $a$ to $z_1$ such that $R^*\subseteq Q_1^*\setminus N_{Q_1}(b)$. Also, since $P_{b,1}\cup Q_2$ is a connected induced subgraph of $G$ containing the two non-adjacent vertices $a$ and $z_1$, it follows that there exists a path $S$ of length at least two in $P_{b,1}\cup Q_2$ from $a$ to $z_1$. But then there is theta in $G$ with ends $a,z_1$ and paths $P_{a,1},R,S$, contrary to the fact that $G\in \mathcal{E}$. This proves \eqref{st:oneofa,b}.
\medskip

From \eqref{st:oneofa,b}, it follows immediately that:

  \sta{\label{st:thetaproperties} Let $i\in \{1,2\}$ such that $z_i$ has a neighbor in $C$. Then the following hold.
  \begin{itemize}
      \item If $z_i$ is $C$-good, then we have $N_C(z_i)\subseteq (N_C(a)\cap N_C(b))\cup \{a,b\}$.
      \item If $z_i$ is $C$-bad, then for some $j\in \{1,2\}$, either $N_C(z_i)=N_{Q_j}[a]$ or $N_C(z_i)=N_{Q_j}[b]$.
      \item If $z_i$ is $C$-ugly, then we have $a,b\in N_C(z_i)$.
  \end{itemize}}

Now, since $C\setminus \{a,b\}=W\setminus P$ and $W\setminus P\subseteq W\setminus P^*=C$, from Lemma~\ref{lem:zintheta}\ref{lem:zintheta_a} and the definition of $G\tri _{z_2}^{z_1}$ it follows that $z_1$ and $z_2$ have no common neighbor in $C$, and from Lemma~\ref{lem:zintheta}\ref{lem:zintheta_c}, it follows that $z_1$ and $z_2$ each have at least one neighbor in $C\setminus \{a,b\}$. Consequently, by Theorem~\ref{thm:evenwheeltheta} and without loss of generality, we may assume that $z_1$ is $C$-bad and $z_2$ is either $C$-good or $C$-ugly. It follows from the second bullet of \eqref{st:thetaproperties} that for some $j\in \{1,2\}$, we have either $N_C(z_1)=N_{Q_j}[a]$ or $N_C(z_1)=N_{Q_j}[b]$. We may exploit the symmetry between $a,b$ and between $Q_1,Q_2$, and assume that $N_{C}(z_1)=N_{Q_1}[a]$. Since $a\in V(H)$ is not a common neighbor of $z_1$ and $z_2$, we deduce from the third bullet of \eqref{st:thetaproperties} that $z_2$ is $C$-good. This, together with the first bullet of \eqref{st:thetaproperties}, the fact that $z_2$ has a neighbor in $C\setminus \{a,b\}$ and the fact that $z_1$ and $z_2$ have no common neighbor in $C$, implies that $Q_2$ has length two, say $Q_2=a\dd q\dd b$, and we have $N_C(z_2)=\{q\}$. But then $G[\{a,q,z_1,z_2
\}]$ is isomorphic to $C_4$, contrary to the fact that $G\in \mathcal{E}$. This completes the proof of Lemma~\ref{lem:<3notheta}.
\end{proof}

\begin{lemma}\label{lem:<3noprism}
   Let $G\in \mathcal{E}$ be a graph and let $z_1,z_2\in V(G)$ be distinct and adjacent such that $N_G(z_1)\cap N_G(z_2)$ is a stable set of vertices of degree at most three in $G$. Then $G\tri^{z_1}_{z_2}$ is prism-free.
\end{lemma}
\begin{proof}
Suppose for a contradiction that there is a prism $W$ in $G\tri^{z_1}_{z_2}$. Let $z\in V(G\tri^{z_1}_{z_2})$ be as in the definition of $G\tri^{z_1}_{z_2}$. Let $P$ be the path in $W$ with ends $a,b$ satisfying Lemma~\ref{lem:zintheta}. It follows from Lemma~\ref{lem:zintheta}\ref{lem:zintheta_a} and \ref{lem:zintheta_b} that $P$ is a path of $W$, $a$ and $b$ belong to distinct triangles of $W$ and we have $z\in P\setminus (N_{P}[a]\cup N_{P}[b])$. Let $aa_1a_2$ and $bb_1b_2$ be the triangles of $W$ and let $Q_1,Q_2$ be the paths of $W$ distinct from $P$ such that $Q_i$ has ends $a_i,b_i$ for $i\in \{1,2\}$. Let $C=Q_1\cup Q_2$. Then $C$ is a hole in $G\setminus \{z_1,z_2
\}$ and we have $C=W\setminus P$. 

 From the definition of $G\tri^{z_1}_{z_2}$, it follows that $W\setminus \{z\}\subseteq G\setminus \{z_1,z_2\}$ and $\{z_1,z_2\}$ is complete to $N_{P}(z)$. As a result, for every $i\in \{1,2\}$, there are two paths $P_{a,i},P_{b,i}$ in $(P\setminus \{z\})\cup \{z_i\}$ from $a$ to $z_i$ and from $b$ to $z_i$, respectively, such that $P_{a,i}\setminus \{z_i\}$ and $P_{b,i}\setminus \{z_i\}$ are disjoint and anticomplete in $G$, and both $P_{a,i}^*$ and $P^*_{b,i}$ are disjoint from and anticomplete to $C$ in $G$. 

Now, since $C=W\setminus P\subseteq W\setminus P^*$, it follows from Lemma~\ref{lem:zintheta}\ref{lem:zintheta_a} and the definition of $G\tri_{z_2}^{z_1}$ that $z_1$ and $z_2$ have no common neighbor in $C$, and it follows from Lemma~\ref{lem:zintheta}\ref{lem:zintheta_c} that $z_1$ and $z_2$ each have at least one neighbor in $C\setminus \{a,b\}$,  Consequently, by Theorem~\ref{thm:evenwheeltheta}, one of $z_1$ and $z_2$ is $C$-bad; say $z_1$ is $C$-bad. Let us write $N_C(z_1)=\{q_1,q_2\}$ where $q_1$ and $q_2$ are adjacent. Due to the symmetry between $\{a_1,a_2\}$ and $\{b_1,b_2\}$, we may assume, without loss of generality, that $\{a_1,a_2\}\cap \{q_1,q_2\}\subseteq  \{a_1\}$, and there are disjoint paths $R_1$ and $R_2$ in $C$ from $a_1$ to $q_1$ and from $a_2$ to $q_2$, respectively. It follows that either $\{a_1,a_2\}\cap \{q_1,q_2\}=\emptyset$ or $\{a_1,a_2\}\cap \{q_1,q_2\}=  \{a_1\}=\{q_1\}$. In the former case, there is a prism in $G$ with triangles $aa_1a_2,z_1q_1q_2$ and paths $P_{a,1},R_1$ and $R_2$. Also, in the latter case, $C'=a\dd P_{a,1}\dd z_1\dd q_2\dd R_2\dd a_2\dd a$ is a hole in $G$ and $a_1=q_1\in G\setminus C'$ has exactly four neighbors in $C'$, namely $a,a_2,q_2$ and $z_1$. But then $(C',z_1)$ is an even wheel in $G$. Both latter conclusions violate the assumption that $G\in \mathcal{E}$, hence completing the proof of Lemma~\ref{lem:<3noprism}.
\end{proof}

We can now give a proof of Theorem~\ref{thm:<3}:

\begin{proof}[Proof of Theorem~\ref{thm:<3}]
   Suppose not. Let $z\in V(G\tri^{z_1}_{z_2})$ be as in the definition of $G\tri^{z_1}_{z_2}$. First, we show that :
     
     \sta{\label{st:c4free} $G\tri^{z_1}_{z_2}$ is $C_4$-free.}

    To see this, suppose there is a hole $C$ of length four in $G\tri^{z_1}_{z_2}$. Since $G\in \mathcal{E}$, it follows that $z\in C$. So we have $N_C(z)=C\cap N_G(z_1)\cap N_G(z_2)$ and there exists exactly one vertex $z'$ in $C$ with $z'\in V(G\tri^{z_1}_{z_2})\setminus N_{G\tri^{z_1}_{z_2}}[z]=V(G)\setminus (N_G[z_1]\cap N_G[z_2])$. As a result, for some $j\in \{1,2\}$,  $z_j$ is not adjacent to $z'$ in $G$. But now $(C\setminus \{z\})\cup \{z_j\}$ is a hole of length four in $G$, contrary to the fact that $G\in \mathcal{E}$. This proves \eqref{st:c4free}.
    \medskip
    
    From \eqref{st:c4free} and Lemmas~\ref{lem:<3notheta} and \ref{lem:<3noprism}, we deduce that there exists an even wheel $(C,v)$ in $G\tri^{z_1}_{z_2}$. Let $W=G[V(C)\cup \{v\}]$ and let $P$ be the path in $W$ with ends $a,b$ satisfying Lemma~\ref{lem:zintheta}. It follows from Lemma~\ref{lem:zintheta}\ref{lem:zintheta_a} and \ref{lem:zintheta_b} that $z\in P\setminus (N_{P}[a]\cup N_{P}[b])$ and $P$ is a path of length at least four in $C$ such that $a,b\in N_C(v)\subseteq N_G(v)$ and $v$ is anticomplete to $P^*$ in $G\tri^{z_1}_{z_2}$. Let $Q=C\setminus P^*$. Then $Q$ is a path in $G$ from $a$ to $b$. Let $a'$ and $b'$ be the neighbors of $a$ and $b$ in $Q$, respectively. Since $C$ is an even wheel in $G\tri^{z_1}_{z_2}$, it follows that $Q$ has length at least three, and so $a,a',b,b'$ are all distinct. In addition, we have $W\setminus P^*=Q\cup \{v\}$, $W\setminus P=Q^*\cup \{v\}$, and $|N_Q(v)|=|N_C(v)|\geq 4$ is an even integer.
    
    From the definition of $G\tri^{z_1}_{z_2}$, it follows that $W\setminus \{z\}\subseteq G\setminus \{z_1,z_2\}$ and $\{z_1,z_2\}$ is complete to $N_{P}(z)$. As a result, for every $i\in \{1,2\}$, there are two paths $P_{a,i},P_{b,i}$ in $(P\setminus \{z\})\cup \{z_i\}$ from $a$ to $z_i$ and from $b$ to $z_i$, respectively, such that $P_{a,i}\setminus \{z_i\}$ and $P_{b,i}\setminus \{z_i\}$ are disjoint and anticomplete in $G$, and both $P_{a,i}^*$ and $P^*_{b,i}$ are disjoint from and anticomplete to $Q^*\cup \{v\}$ in $G$. We claim that:

 \sta{\label{st:linewheel} The vertex $v$ has a neighbor in $Q\setminus \{a,a',b,b'\}$.}

    Suppose not. Then we have $N_Q(v)= \{a,a',b,b'\}$. Then $C'=Q^*\cup \{v\}=W\setminus P$ is a hole in $G$. By Lemma~\ref{lem:zintheta}\ref{lem:zintheta_a} and the definition of $G\tri _{z_2}^{z_1}$, $z_1$ and $z_2$ have no common neighbor in $C'$, and by Lemma~\ref{lem:zintheta}\ref{lem:zintheta_c}, $z_1$ and $z_2$ each have at least one neighbor in $C'$.  Therefore, by Theorem~\ref{thm:evenwheeltheta}, one of $z_1$ and $z_2$, say the former, is $C'$-bad in $G$. Let $N_{C'}(z_1)=\{q,q'\}$ where $q$ and $q'$ are adjacent. The symmetry between $a'$ and $b'$ and between $q$ and $q'$ allows us to assume that $|\{a',v\}\cap \{q,q'\}|\leq 1$, and there are disjoint paths $R$ and $R'$ in $C'$ from $v$ to $q$ and from $a'$ to $q'$, respectively. It follows that either 
    \begin{itemize}
        \item $\{a',v\}\cap \{q,q'\}=\emptyset$; or 
        \item $\{a',v\}\cap \{q,q'\}=  \{a'\}=\{q'\}$; or 
        \item $\{a',v\}\cap \{q,q'\}=  \{v\}=\{q\}$.
    \end{itemize}
    If the first bullet above holds, then there is a prism in $G$ with triangles $aa'v, qq'z_1$ and paths $P_{a,1},R$ and $R'$. Also, if the second bullet above holds, then $C''=a\dd P_{a,1}\dd z_1\dd q\dd R\dd v\dd a$ is a hole in $G$ and $a'=q'\in G\setminus C''$ has exactly four neighbors in $C''$, namely $a,v,q$ and $z_1$, which in turn implies that $(C'',a')$ is an even wheel in $G$. Similarly, if the third bullet above holds,  then $C''=a\dd P_{a,1}\dd z_1\dd q'\dd R'\dd a'\dd a$ is a hole in $G$ and $v=q\in G\setminus C''$ has exactly four neighbors in $C''$, namely $a,a',q'$ and $z_1$. It follows that $(C'',v)$ is an even wheel in $G$.
    Each of the last three conclusions goes against the assumption that $G\in \mathcal{E}$. This proves \eqref{st:linewheel}.

    \sta{\label{st:ztohub} Let $i\in \{1,2\}$ such that $v$ is not adjacent to $z_i$ in $G$. Then $N_{W\setminus P}(z_i)$ is a non-empty subset of $\{a',b'\}$.}

    Suppose not.  By Lemma~\ref{lem:zintheta}\ref{lem:zintheta_c}, $z_1$ and $z_2$ each have at least one neighbor in $
    W\setminus P=Q^*\cup \{v\}$, and so $N_{W\setminus P}(z_i)\neq \emptyset$. It follows that there exists a vertex $q\in Q^*\setminus \{a',b'\}=Q\setminus (N_G[a]\cup N_G[b])$ which is adjacent to $z_i$ in $G$. This, together with \eqref{st:linewheel}, implies that $(Q^*\setminus \{a',b'\})\cup \{v,z_i\}$ is connected, and so there exists a path $S$ of length at least two in $(Q^*\setminus \{a',b'\})\cup \{v,z_i\}$ from $v$ to $z_i$. But now there is a theta in $G$ with ends $v,z_i$ and paths $v\dd a\dd P_{a,i}\dd z_i,v\dd b\dd P_{b,i}\dd z_i$ and $S$, contrary to the fact that $G\in \mathcal{E}$. This proves \eqref{st:ztohub}.

  \sta{\label{st:vnotanti} There exists $i\in \{1,2\}$ for which $v$ is adjacent to $z_i$ in $G$.}

   Suppose for a contradiction that $v$ is anticomplete $\{z_1,z_2\}$. By \eqref{st:ztohub}, both $N_{W\setminus P}(z_1)$ and $N_{W\setminus P}(z_2)$ are non-empty subset of $\{a',b'\}$. By Lemma~\ref{lem:zintheta}\ref{lem:zintheta_a} and the definition of $G\tri_{z_2}^{z_1}$, in the graph $G$, $z_1$ and $z_2$ do not have a common neighbor in $
    W\setminus P^*=Q\cup \{v\}$. Therefore, due to the symmetry between $z_1$ and $z_2$ and between $a'$ and $b'$, it is safe to assume that $N_{Q^*\cup \{v\}}(z_1)=\{a'\}$ and $N_{Q^*\cup \{v\}}(z_2)=\{b'\}$. In particular, $D=a'\dd z_1\dd z_2\dd b'\dd Q\dd a'$ is a hole in $G$ and $N_D(v)=N_Q(v)\setminus \{a,b\}$. Recall that $|N_Q(v)|$ is an even integer which is at least four, and so $|N_D(v)|$ is a non-zero even integer. Since $(D,v)$ is not even wheel in $G$ and $D\cup \{v\}$ is not a theta in $G$, it follows that $v$ is $D$-bad. From this combined with \eqref{st:linewheel}, and without loss of generality, we may assume that $v$ is not adjacent to $a'$ and there is a path $R$ in $G$ from $a'$ to $v$ with $R^*\subseteq Q^*\setminus \{a',b'\}$. Moreover, again by Lemma~\ref{lem:zintheta}\ref{lem:zintheta_a} and the definition of $G\tri_{z_2}^{z_1}$, in the graph $G$, $z_1$ and $z_2$ do not have a common neighbor in $
    W\setminus P^*=Q\cup \{v\}$. Specifically, there exists $i\in \{1,2\}$ for which $z_i$ is no adjacent to $a$. But now there is a theta in $G$ with ends $a,z_i$ and paths $a\dd a'\dd z_i,P_{a,i}$ and $a\dd v\dd b\dd P_{b,i}\dd z_i$, a contradiction. This proves \eqref{st:vnotanti}.

    \sta{\label{st:bothztohub} Let $i\in \{1,2\}$ such that $v$ is not adjacent to $z_i$ in $G$. Then either $z_i$ is anticomplete to $\{a,b\}$ in $G$, or $v$ is anticomplete to $\{z_1,z_2\}$ in $G$.}

    Suppose not. By Lemma~\ref{lem:zintheta}\ref{lem:zintheta_a}, $z_1$ and $z_2$ do not have a common neighbor in $\{a,b,v\}\subseteq Q\cup \{v\}=
    W\setminus P^*$. But now either either $G[\{a,v,z_1,z_2\}]$ or $G[\{b,v,z_1,z_2\}]$ is isomorphic to $C_4$, contrary to the fact that $G\in \mathcal{E}$. This proves \eqref{st:bothztohub}.
    \medskip

    Let us now finish the proof. In view of \eqref{st:vnotanti}, we may assume, without loss of generality, that $v$ is adjacent to $z_2$ in $G$. Thus, by \eqref{st:bothztohub}, $z_1$ is anticomplete to $\{a,b\}$ in $G$. Since $N_{W\setminus P}(z_1)$ is a non-empty subset of $\{a',b'\}$, we may assume that $a'$ is adjacent to $z_1$ in $G$. Moreover, since $G[\{a',v,z_1,z_2\}]$ is not isomorphic to $C_4$, it follows that $a'$ and $v$ are not adjacent in $G$. But then there is a theta in $G$ with ends $a,z_1$ and paths $a\dd a'\dd z_1,P_{a,1}$ and $a\dd v\dd b\dd P_{b,1}\dd z_1$, contrary to the fact that $G\in \mathcal{E}$. This completes the proof of Theorem~\ref{thm:<3}.
    \end{proof}

\section{Third time is the charm}\label{sec:gbu}

This section supplies the technical foundation for the proofs of Theorems~\ref{thm:main2-tree} and \ref{thm:tree+}. Despite its crucial role, the proof of Theorem~\ref{thm:difftodiff} below as the main result revolves around three successive applications of Theorem~\ref{banana}. Perhaps more importantly, it highlights the significance of not giving up after the first two rounds.

In essence, Theorem~\ref{thm:difftodiff} asserts that in a (theta, prism, even wheel)-free graph of bounded clique number with a configuration of sufficiently many ``put-together'' holes, a vertex $z$ with at least one private neighbor in each hole will have \textit{several} neighbors in most of these holes. This result, combined with Theorem~\ref{thm:evenwheeltheta}, will be used in Section~\ref{sec:kaleidoscope} to demonstrate that adjacent pairs of vertices with neighbors in these holes must indeed have common neighbors within them.

To make this precise, we begin with the definition of the said configuration of holes. Given a graph $G$ and an integer $w\geq 1$, a \textit{$w$-kaleidoscope} in $G$ is a 4-tuple $(a,x,y,\mathcal{W})$ where:
\begin{enumerate}[(K1), leftmargin=15mm, rightmargin=7mm]
\item\label{K1} $a,x,y\in V(G)$, and $x\dd a\dd y$ is a path in $G$ (so $x$ and $y$ are distinct and non-adjacent);
\item\label{K2} $\mathcal{W}$ is a set of $w$ pairwise internally disjoint paths in $G\setminus {a}$ from $x$ to $y$; and
\item\label{K3} for every $W\in \mathcal{W}$, $a$ is anticomplete to $W^*$ in $G$.
\end{enumerate}

Furthermore, given a subset $Z\subseteq V(G)$ and an integer $d\geq 1$, we say that $Z$ \textit{is $d$-mirrored by $(a,x,y,\mathcal{W})$} if:
\begin{enumerate}[(M1), leftmargin=15mm, rightmargin=7mm]
\item\label{M1} $Z$ is disjoint from $(\bigcup_{W\in \mathcal{W}}V(W))\cup \{a\}$;
\item\label{M2} the vertex $a$ has at most one neighbor in $Z$; and
\item\label{M3} for every $z\in Z$ and every $W\in \mathcal{W}$, $z$ is anticomplete to $N_{W}[x]\cup N_{W}[y]$, and $z$ has at least $d$ distinct neighbors in $W$. In particular, $z$ is anticomplete to $\{x,y\}$.
\end{enumerate}
We also say a vertex $z\in V(G)$ is \textit{$d$-mirrored by $(a,x,y,\mathcal{W})$} if $\{z\}$ is $d$-mirrored by $(a,x,y,\mathcal{W})$. Our goal in this section is to show that:
\begin{theorem}\label{thm:difftodiff}
For all integers $d,t,w\geq 1$, there exists an integer $\kappa=\kappa(d,t,w)\geq 1$ with the following property. Let $G$ be a (theta, prism, even wheel, $K_t$)-free graph, let $(a,x,y,\mathcal{W})$ be a $\kappa$-kaleidoscope in $G$, and let $z\in V(G)$ be $1$-mirrored by $(a,x,y,\mathcal{W})$. Then there exists $\mathcal{W}'\subseteq \mathcal{W}$ with $|\mathcal{W}'|=w$ such that $z$ is $d$-mirrored by $(a,x,y,\mathcal{W}')$.
\end{theorem}

We start with a number of further definitions and lemmas. Let $G$ be a graph, and let $s,l\geq 1$ be integers. An \textit{$(s,l)$-palanquin} in $G$ is a triple $(a,S,\mathcal{L})$ where:

\begin{enumerate}[(P1), leftmargin=15mm, rightmargin=7mm]
\item\label{P1} $a\in V(G)$, $S\subseteq N_G(a)$ is a stable set of cardinality $s$ in $G$, and $\mathcal{L}$ is a collection of $l$ pairwise disjoint paths in $G\setminus (S\cup \{a\})$; and
\item\label{P2} for every $L\in \mathcal{L}$, $a$ is anticomplete to $L$, and every vertex in $S$ has a neighbor in $L$.
\end{enumerate}
For instance, given a $w$-kaleidoscope $(a,x,y,\mathcal{W})$ in $G$ for some $w\geq 1$, one may easily observe that $(a,\{x,y\},\{W^*:W\in\mathcal{W}\})$ is a $(2,w)$-palanquin in $G$.
Next we have two lemmas about palanquins, with short and self-contained proofs (and, although less efficiently, these lemmas can also be deduced from appropriate results in earlier papers of this series).

\begin{lemma}\label{lem:nocommon&good}
    Let $p,q,t\geq 1$ be integers. Let $G$ be a (theta, $K_t$)-free graph and let $(a,S,\mathcal{L})$ be a $(2^qp+2q,(2^qp+2q)^2t^3+q)$-palanquin in $G$. Then there exists $S_1\subseteq S$ with $|S_1|=p$ and $\mathcal{L}_1\subseteq \mathcal{L}$ with $|\mathcal{L}_1|=q$ such that for every $L\in \mathcal{L}_1$, no two vertices in $S_1$ have a common neighbor in $L$, and either $x$ is $L$-bad for all $x\in S_1$, or $x$ is $L$-ugly for all $x\in S_1$.
\end{lemma}
\begin{proof}
    We first show that:

    \sta{\label{st:nocommon} There exists $\mathcal{L}_0\subseteq \mathcal{L}$ with $|\mathcal{L}_0|=q$ such that for every $L\in \mathcal{L}_0$, no two vertices in $S$ have a common neighbor in $L$.}

Suppose not. Since $|\mathcal{L}|=(2^qp+2q)^2t^3+q$, it follows that there exists $\mathcal{L}'_1\subseteq \mathcal{L}$ with $|\mathcal{L}'_1|=(2^qp+2q)^2t^3$ such that for every $L\in \mathcal{L}'_1$, there exist two distinct vertices in $S$ with a common neighbor in $L$. Since $|S|=2^qp+2q$, this in turn implies that there are two vertices $x,x'\in S$ as well as a subset $\mathcal{L}''_1$ of $\mathcal{L}'_1$ with $|\mathcal{L}''_1|=t^3$ such that for every $L\in \mathcal{L}''_1$, the vertices $x,x'$ have a common neighbor $y_L$ in $L$. Thus, we have $|\{y_L:L\in \mathcal{L}''_1\}|=t^3$, which along with Theorem~\ref{classicalramsey} and the assumption that $G$ is $K_t$-free implies that there are three distinct paths $L_1,L_2,L_3\in \mathcal{L}''_1$ for which $\{y_{L_1},y_{L_2},y_{L_3}\}$ is a stable set in $G$. But now there is a theta in $G$ with ends $x,x'$ and paths $x\dd y_{L_1}\dd x'$, $x\dd y_{L_2}\dd x'$ and $x\dd y_{L_3}\dd x'$, a contradiction. This proves \eqref{st:nocommon}.

\medskip

Henceforth, let $\mathcal{L}_0$ be as in \eqref{st:nocommon}.

\sta{\label{sta:killgood} There exists $S_0\subseteq S$ with $|S_0|=2^{q}p$ such that for every $x\in S_0$ and every $L\in \mathcal{L}_0$, $x$ is either $L$-bad or $L$-ugly.}

Note that every vertex in $S$ has a neighbor in every path in $\mathcal{L}_0\subseteq \mathcal{L}$. Therefore, since $|S|=2^{q}p+2q$ and $|\mathcal{L}_0|=q$, in order to prove \eqref{sta:killgood}, it suffices to show that for every $L\in \mathcal{L}_0$, there at most two vertices in $S$ which are $L$-good. Suppose for a contradiction that there exists a $3$-subset $\{x_1,x_2,x_3\}$ of $S$ and a path $L\in \mathcal{L}_0$ such that $x_1,x_2,x_3$ are all $L$-good. For each $i\in [3]$, let $y_i$ be the unique neighbor of $x_i$ in $L$. We may assume without loss of generality that the path in $L$ from $y_1$ to $y_3$ contains $y_2$. Recall also that $a$ is complete to $\{x_1,x_2,x_3\}\subseteq S$
and anticomplete to $L\in \mathcal{L}_0\subseteq\mathcal{L}$. But now there is a theta in $G$ with ends $a,y_2$ and paths $a\dd x_1\dd y_1\dd L\dd y_2, a\dd x_2\dd y_2$ and $a\dd x_3\dd y_3\dd L\dd y_2$, a contradiction. This proves \eqref{sta:killgood}.

\medskip

Let $S_0$ be as in \eqref{sta:killgood}. Fix an enumeration $L_1,\ldots, L_{q}$ of the elements of $\mathcal{L}_0$. In view of \eqref{sta:killgood}, one may construct a sequence $X_0\supset X_1\supset \cdots\supset X_{q}$ of sets such that:
\begin{itemize}
    \item $X_0=S_0$ and for every $i\in [q]$, we have $|X_i|=|X_{i-1}|/2$; and
    \item for every $i\in [q]$, either $x$ is $L_i$-bad for all $x\in X_i$ or $x$ is $L_i$-ugly for all $x\in X_i$.
\end{itemize}
Let $S_1=X_{q}$. Then we have $|S_1|=|X_{q}|=2^{-q}|X_0|=2^{-q}|S_0|=p$, and for every $i\in [q]$, either $x$ is $L_i$-bad for all $x\in S_1$ or $x$ is $L_i$-ugly for all $x\in S_1$. But then by \eqref{st:nocommon}, $S_1$ and $\mathcal{L}_1$ satisfy Lemma~\ref{lem:nocommon&good}, as required.
\end{proof}
{%
\begin{figure}[t!]
  \centering
  \includegraphics[scale=0.7]{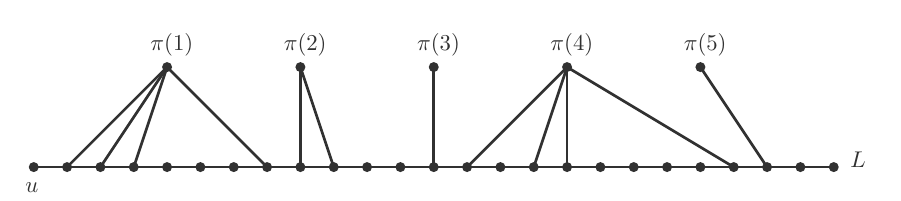}
  \captionof{figure}{A $5$-alignment, where $\pi(3)$ and $\pi(5)$ are $L$-good, $\pi(2)$ is $L$-bad and $\pi(1)$ and $\pi(4)$ are $L$-ugly.}
  \label{fig:alignment}
  \end{figure}
}
Let $G$ be a graph and let $s\geq 1$ be an integer. An \textit{$s$-alignment} is a quadruple $(S,L,x,\pi)$ where:

\begin{enumerate}[({\rm A}1), leftmargin=15mm, rightmargin=7mm]
\item\label{A1} $S\subseteq V(G)$ is stable with $|S|=s$ and $L$ is a path in $G\setminus S$ and $x$ is an end of $L$;
\item\label{A2} every vertex in $S$ has a neighbor in $L$; and
\item\label{A3} $\pi:[s]\rightarrow S$ is a bijection such that for all $i,j\in [s]$ with $i<j$, traversing $L$ starting at $x$, all neighbors of $\pi(i)$ appear strictly  before all neighbors of $\pi(j)$ in $L$.
\end{enumerate}
See Figure~\ref{fig:alignment}. It follows in particular from \ref{A3} that no two vertices in $S$ have a common neighbor in $L$.

\begin{lemma}\label{lem:getalignment}
    Let $s\geq 1$ be an integer. Let $G$ be a theta-free graph, let $(a,S,\{L\})$ be an $(s,1)$-palanquin in $G$ and let $x_L$ be an end of $L$. Assume that no two vertices in $S$ have a common neighbor in $L$. Assume also that either $x$ is $L$-bad for all $x\in S$, or $x$ is $L$-ugly for all $x\in S$. Then there exists a bijection $\pi:S\rightarrow [s]$ such that $(S,L,x_L,\pi)$ is an $s$-alignment in $G$.
    \end{lemma}
\begin{proof}
Since no two vertices in $S$ have a common neighbor in $L$, the result is trivial if every vertex in $S$ is $L$-bad. So we may assume that all vertices $S$ are $L$-ugly. For every $x\in S$, traversing $L$ starting at $x_L$, let $u_x,v_x$ be the first and the last neighbor of $x$ in $L$, respectively; thus, $u_x$ and $v_x$ are distinct. In order to prove Lemma~\ref{lem:getalignment}, it suffices to show that for every two vertices $x_1,x_2\in S$, the paths $u_{x_1}\dd L\dd v_{x_1}$ and $u_{x_2}\dd L\dd v_{x_2}$ have no vertex in common. Suppose this is violated by  $x_1,x_2\in S$. Since $x_1$ and $x_2$ have no common neighbor in $L$, it follows that $u_{x_1},u_{x_2},v_{x_1},v_{x_2}$ are pairwise distinct, $x_1$ is not adjacent to $u_{x_2},v_{x_2}$, and $x_2$ is not adjacent to $u_{x_1},v_{x_1}$. Since both $x_1$ and $x_2$ are $L$-ugly, we may assume without loss of generality that for some $\{i,j\}=\{1,2\}$, $L$ traverses the vertices  $x_L, u_{x_1},u_{x_2},v_{x_i},v_{x_j}$ in this order (where $x_L$ and $u_{x_1}$ might be the same). It follows that there are two paths $P$ and $Q$ in $G$ from $x_1$ to $x_2$ such that $P^*$ contains $u_{x_2}$ and is contained in $u_{x_1}\dd L\dd u_{x_2}$, and $Q^*$ contains $v_{x_i}$ and is contained in $v_{x_1}\dd L\dd v_{x_2}$. In particular, $P$ and $Q$ are internally disjoint. Recall also that $a$ is complete to $\{x_1,x_2\}\subseteq S$
and anticomplete to $L$. Now, if the path $u_{x_2}\dd L\dd v_{x_i}$ has non-empty interior, then $P^*$ and $Q^*$ are anticomplete, and so there is a theta in $G$ with ends $x_1,x_2$ and paths $x_1\dd a\dd x_2, P$ and $Q$, a contradiction (see Figure~\ref{fig:thetaonpath} top). Otherwise, since $x_2$ is $L$-ugly, we have $i=1$, and so there is a theta in $G$ with ends $x_1,u_{x_2}$ and paths $x_1\dd a\dd x_2\dd u_{x_2}, x_1\dd P\dd u_{x_2}$ and $x_1\dd v_{x_1}\dd u_{x_2}$, again a contradiction (see Figure~\ref{fig:thetaonpath} bottom). This completes the proof of Lemma~\ref{lem:getalignment}.
\end{proof}

We apply Lemmas~\ref{lem:nocommon&good} and \ref{lem:getalignment} to take the main step in the proof of Theorem~\ref{thm:difftodiff}. This is where two of the three applications of Theorem~\ref{banana} show up (and the third one will appear at the beginning of the proof of Theorem~\ref{thm:difftodiff}).

\begin{lemma}\label{lem:palanquintoalignment}   

For all integers $s,l,t\geq 1$, there exist integers $\sigma=\sigma(l,s,t)\geq 1$ and $\lambda=\lambda(l,s,t)\geq 1$ with the following property. Let $G$ be a (theta, prism, even wheel, $K_t$)-free graph. Let $(a,S,\mathcal{L})$ be an $(\sigma,\lambda)$-palanquin in $G$. For every $L\in \mathcal{L}$, fix an end $x_L$ of $L$. Then there exist $S'\subseteq S$ with $|S'|=s$, an $l$-subset $\mathcal{L}'$ of $\mathcal{L}$ and a bijection $\pi:S'\rightarrow [s]$ such that the following hold.

\begin{enumerate}[\rm (a)]
    \item\label{lem:palanquintoalignment_a} For every $L\in \mathcal{L}'$, the quadruple $(S',L,x_L, \pi)$ is an $s$-alignment in $G$.
     \item\label{lem:palanquintoalignment_b} For every $x\in S'$ and every $L\in \mathcal{L}'$, the vertex $x$ is $L$-ugly.

 \end{enumerate} 
\end{lemma}

{%
\begin{figure}[t!]
  \centering
  \includegraphics[scale=0.7]{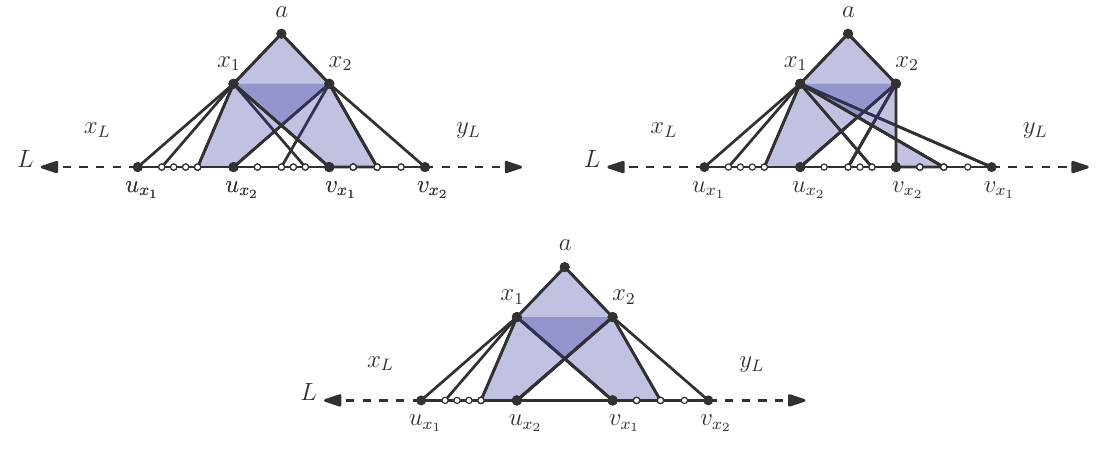}
  \captionof{figure}{Proof of Lemma~\ref{lem:getalignment}. Top left: the case $i=1$. Top right: the case $i=2$.}
  \label{fig:thetaonpath}
  \end{figure}
}
    
\begin{proof}
We begin with defining the values of $\sigma$ and $\lambda$. Let $\psi(\cdot,\cdot)$ be as in Theorem~\ref{banana} and let $o(\cdot,\cdot,\cdot,\cdot)$ be as in Lemma~\ref{ramsey2}. The two successive applications of Theorem~\ref{banana} are already signaled by the two nested appearances of $\psi(\cdot,\cdot)$ below. Let
$$\psi_1=\psi_1(t)=\psi(t,2);$$
$$o_1=o_1(t)=o(\psi_1,4,3,t);$$
$$\psi_2=\psi_2(t)=\psi(t,o_1+1);$$
$$o_2=o_2(t)=o(\psi_2,2,3,t);$$
Also, let $p=p(s)=s+2$ and
let $q=q(l,s,t)=(l+o_2)(s+2)!$.
We claim that $\sigma=\sigma(l,s,t)=2^qp+2q$ and $\lambda=\lambda(l,s,t)=q(2^qp+2q)^2t^3$
satisfy Lemma~\ref{lem:palanquintoalignment}. 

Suppose not. Due to the choice of $\sigma$ and $\lambda$, we can apply Lemma~\ref{lem:nocommon&good} to $(a,S,\mathcal{L})$, obtaining $S_1\subseteq S$ with $|S_1|=s+2$ and $\mathcal{L}_1\subseteq \mathcal{L}$ with $|\mathcal{L}_1|=(l+o_2)(s+2)!$ such that for every $L\in \mathcal{L}_1$, no two vertices in $S_1$ have a common neighbor in $L$, and either $x$ is $L$-bad for all $x\in S_1$, or $x$ is $L$-ugly for all $x\in S_1$. This, combined with Lemma~\ref{lem:getalignment}, implies that for every $L\in \mathcal{L}_1$, there exists a bijection $\pi_L:S_1\rightarrow [s+2]$ such that $(S_1,L,x_L,\pi_L)$ is an $(s+2)$-alignment in $G$. We deduce that:

\sta{\label{st:getfit} There exists $\mathcal{L}_2\subseteq \mathcal{L}_1$ with $|\mathcal{L}_2|=o_2(s+2)!$ such that for every $x\in S_1$ and every $L\in \mathcal{L}_1$, $x$ is $L$-bad.}

Suppose not. From $|\mathcal{L}_1|=(l+o_2)(s+2)!$, it follows that there exists $\mathcal{L}'_2\subseteq \mathcal{L}_1$ with $\mathcal{L}'_2=l(s+2)!$ such that for every $x\in S_1$ and every $L\in \mathcal{L}'_2$, $x$ is $L$-ugly. Since $|S_2|=s+2$, this in turn implies that there exists $\mathcal{L}'\subseteq \mathcal{L}'_2$ with $|\mathcal{L}'|=l$ as well as a bijection  $\pi':S_2\rightarrow [s+2]$, such that $\pi_L=\pi'$ for all $L\in \mathcal{L}'$. Let $S'=\pi'([s])$ and let $\pi=\pi'|_{[s]}$. Then for every $L\in \mathcal{L}'$, $(S',L,x_L,\pi)$ is an $s$-alignment in $G$, and so $S',\mathcal{L}'$ satisfy \ref{lem:palanquintoalignment}\ref{lem:palanquintoalignment_a}. Also, since $S'\subseteq S_1$ and $\mathcal{L}'\subseteq \mathcal{L}'_2$, it follows that $S',\mathcal{L}'$ satisfy \ref{lem:palanquintoalignment}\ref{lem:palanquintoalignment_b}. This violates the assumption that Lemma~\ref{lem:palanquintoalignment} fails to hold for our chosen values of $\sigma$ and $\lambda$, hence proving \eqref{st:getfit}.

\medskip

Let $\mathcal{L}_2$ be as in \eqref{st:getfit}. Since $|S_1|=s+2$ and $|\mathcal{L}_2|=o_2(s+2)!$, it follows that there exists $\mathcal{L}_3\subseteq \mathcal{L}_2$ with $|\mathcal{L}_3|=o_2$ as well as a bijection  $\pi:S_1\rightarrow [s+2]$, such that $\pi_L=\pi$ for all $L\in \mathcal{L}_3$. Let us write $x=\pi(1), z=\pi(2)$ and $y=\pi(3)$ (this is possible because $s+2\geq 3$).

For every $L\in \mathcal{L}_3$, traversing $L$ starting at $x_L$, let $u_L$ be the last neighbor of $x$ in $L$, let $z^1_L, z^2_L$ be the first and the last neighbor of $z$ in $L$, respectively, and let $v_L$
be first neighbor of $y$ in $L$. By \eqref{st:getfit}, $u_L,z_L^1,z_L^2$ and $v_L$ are all distinct, appearing on $L$ in this order, and $N_L(z)=\{z^1_L,z_L^2\}$ is a clique in $G$. Since $G$ is $(K_{3,3},K_t)$-free and from the choice of $o_2$, it follows that we can apply Lemma~\ref{ramsey2} to the sets $\{\{z^1_L,z_L^2\}:L\in \mathcal{L}_3\}$ and show that:

\sta{\label{st:roundoneanti} There exists $\mathcal{L}'_3\subseteq \mathcal{L}_3$ with $|\mathcal{L}'_3|=\psi_2$ such that for all distinct $L,L'\in \mathcal{L}'_3$, $\{z^1_L,z_L^2\}$ is anticomplete to $\{z^1_{L'},z_{L'}^2\}$.}

Next, we launch the first application of Theorem~\ref{banana}. Note that $\{z\dd z_L^2\dd L\dd v_L\dd y:L\in \mathcal{L}'_3\}$ is a collection of $\psi_2$ pairwise internally disjoint paths in $G$ between non-adjacent vertices $z$ and $y$. Consequently, due to the choice of $\psi_2$, we can apply Theorem~\ref{banana} to this collection, and deduce that there exists $L_3\in \mathcal{L}'_3$ as well as $\mathcal{L}_4\subseteq \mathcal{L}'_3\setminus \{L_3\}$ with $|\mathcal{L}_4|=o_1$ such that:
\begin{itemize}
    \item $\{z^2_L:L\in \mathcal{L}_4\}\cup \{z^2_{L_3},y\}$ is a stable set in $G$ (though this is already guaranteed by \eqref{st:roundoneanti}); and
    \item for all $L\in \mathcal{L}_4$, $z^2_{L_3}$ has a neighbor in the interior of $z_L^2\dd L\dd v_L\dd y$.
    \end{itemize}
    
   For each $L\in \mathcal{L}_4$, traversing $z_L^2\dd L\dd v_L\dd y$ from $z^2_L$ to $y$, let $w^1_L,w^2_L$ be the first and the last neighbors of $z^2_{L_3}$ in $z_L^2\dd L\dd v_L\dd y$, respectively; it follows that $\{w^1_L,w^2_L\}\cap\{z^2_L,y\}=\emptyset$ and  $z^1_L,z^2_L,w^1_L,w^2_L$ appear on $L$ in this order. For every $L\in \mathcal{L}_4$, let $C_L$ denote the hole $a\dd x\dd u_L\dd L\dd v_L\dd y\dd a$ in $G$. We deduce that:

    \sta{\label{st:round2fit} For every $L\in \mathcal{L}_4$, $z^2_{L_3}$ is $C_L$-bad.  More explicitly, $w^1_L$ and $w^2_L$ are distinct and adjacent, and we have $N_{C_L}(z^2_{L_3})=\{w^1_L,w^2_L\}$.}

To see this, note that $z$ and $z^2_{L_3}$ are two adjacent vertices in $G\setminus C_L$, each with at least one neighbor in $C_L$. In fact, we have $N_{C_L}(z)=\{a,z^1_L,z^2_L\}$, and so $z$ is ${C_L}$-ugly. Since $a$ is anticomplete to $L_3$ and from \eqref{st:roundoneanti}, it follows that $z^2_{L_3}$ is anticomplete to $\{a,z^1_L,z^2_L\}=N_{C_L}(z)$. Thus, $z$ and $z^2_{L_3}$ have no common neighbor in $C_L$. So by Theorem~\ref{thm:evenwheeltheta}, $z^2_{L_3}$ is $C_L$-bad. This proves \eqref{st:round2fit}.

\medskip

Furthermore, since $G$ is $(K_{3,3},K_t)$-free and from the choice of $o_1$, we can apply Lemma~\ref{ramsey2} to the sets $\{\{w^1_L,w_L^2,z^1_L,z_L^2\}:L\in \mathcal{L}_4\}$, and deduce that:

\sta{\label{st:roundtwoanti} There exists $\mathcal{L}'_4\subseteq \mathcal{L}_4$ with $|\mathcal{L}'_4|=\psi_1$ such that for all distinct $L,L'\in \mathcal{L}'_4$, $\{w^1_L,w_L^2,z^1_L,z_L^2\}$ is anticomplete to $\{w^1_{L'},w_{L'}^2,z^1_{L'},z_{L'}^2\}$.}

Now, note that $\{z^2_{L_3}\dd w_L^2\dd L\dd v_L\dd y:L\in \mathcal{L}'_4\}$ is a collection of $\psi_1$ pairwise internally disjoint paths in $G$ between non-adjacent vertices $z^2_{L_3}$ and $y$. Together with the choice of $\psi_1$, this allows an application of Theorem~\ref{banana} to $\{z^2_{L_3}\dd w_L^2\dd L\dd v_L\dd y:L\in \mathcal{L}'_4\}$. We obtain two distinct paths $L_1,L_2\in \mathcal{L}'_4$ such that
\begin{itemize}
    \item $\{w^2_{L_1},w^2_{L_2},y\}$ is a stable set in $G$ (though this already follows from \eqref{st:roundtwoanti}); and
    \item the vertex $w^2_{L_2}$ has a neighbor in the interior of $w_{L_1}^2\dd L_1\dd v_L\dd y$.
    \end{itemize}
    
   Traversing $w_{L_1}^2\dd L_1\dd v_L\dd y$ from $w_{L_1}^2$ to $y$, let $w,w'$ be the first and the last neighbors of $w^2_{L_2}$ in $w_{L_1}^2\dd L_1\dd v_L\dd y$, respectively; it follows that $\{w,w'\}\cap\{w_{L_1}^2,y\}=\emptyset$ and  $w_{L_1}^1,w_{L_1}^2,w,w'$ appear on $L_1$ in this order. In addition, we have:
{
\begin{figure}[t!]
\centering

\includegraphics[scale=0.8]{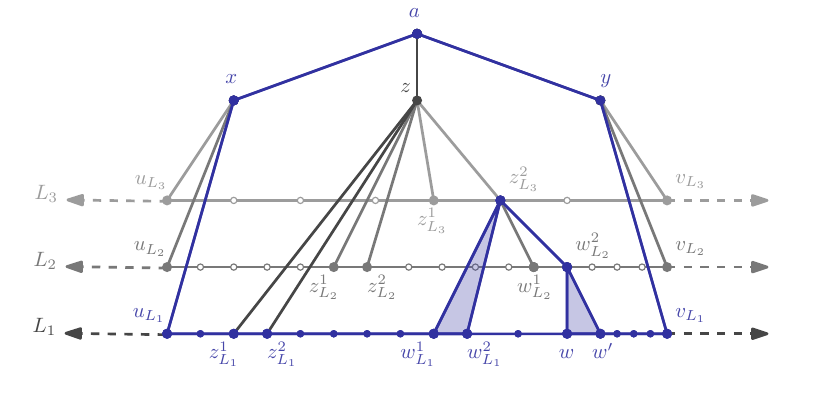}

\caption{Proof of Lemma~\ref{lem:palanquintoalignment}.}
\label{fig:lem:palanquintoalignment}
\end{figure}
}

    \sta{\label{st:round3fit} The vertex $w^2_{L_2}$ is $C_{L_1}$-bad.  More precisely, $w$ and $w'$ are distinct and adjacent, and we have $N_{C_{L_1}}(w^2_{L_2})=\{w,w'\}$.}

Let $C_1=a\dd z\dd z^2_{L_1}\dd L_1\dd v_L\dd y\dd a$; then $C_1$ is a hole in $G$. Note that $w^2_{L_2}$ and $z^2_{L_3}$ are two adjacent vertices in $G\setminus C_1$, each with at least one neighbor in $C_1$. In fact, we have $N_{C_1}(z^2_{L_3})=\{z,w^1_{L_1},w^2_{L_1}\}$, and so $z^2_{L_3}$ is ${C_1}$-ugly. This, along with \eqref{st:roundtwoanti}, implies that $w^2_{L_2}$ is anticomplete to $\{z,w^1_{L_1},w^2_{L_1}\}=N_{C_1}(z^2_{L_3})$. It follows that $w^2_{L_2}$ and $z^2_{L_3}$ have no common neighbor in $C_1$. But then by Theorem~\ref{thm:evenwheeltheta}, $w^2_{L_2}$ is $C_1$-bad. More precisely, $w$ and $w'$ are distinct and adjacent, and we have $N_{C_{1}}(w^2_{L_2})=\{w,w'\}$. Since $C_{L_1}\setminus C_1=x\dd u_{L_1}\dd L_1\dd z^1_{L_1}$, it remains to show that $w^2_{L_2}$ is anticomplete to $u_{L_1}\dd L_1\dd z^1_{L_1}$. Suppose not. Recall that by \eqref{st:roundtwoanti}, $a$ and $w^2_{L_2}$ are not adjacent in $G$. Consequently, there is a path $Q$ of length at least two in $G$ from $a$ to $w^2_{L_2}$ such that $Q^*$ is contained in the interior of $a\dd x\dd  u_{L_1}\dd L_1\dd z^1_{L_1}$. But then in view of \eqref{st:round2fit}, there is a theta in $G$ with ends $a,w^2_{L_2}$ and paths $a\dd z\dd z^2_{L_3}\dd w^2_{L_2}, a\dd y\dd v_{L_1}\dd L_1\dd w'\dd w^2_{L_2}$ and $Q$, a contradiction. This proves \eqref{st:round3fit}.

\medskip
Finally, by \eqref{st:round2fit} and \eqref{st:round3fit}, there is a prism in $G$ with triangles $z^2_{L_3}w^1_{L_1}w^2_{L_1}$ and $w^2_{L_2}ww'$ and paths $z^2_{L_3}\dd w^2_{L_2}$, $w^2_{L_1}\dd L_1\dd w$ and $w^1_{L_1}\dd L_1\dd u_{L_1}\dd x\dd a\dd y\dd v_{L_1}\dd L_1\dd w'$ (see Figure~\ref{fig:lem:palanquintoalignment}), a contradiction. This completes the proof of Lemma~\ref{lem:palanquintoalignment}.
\end{proof}

We are now in a position to prove Theorem~\ref{thm:difftodiff}:
    
    \begin{proof}[Proof of Theorem~\ref{thm:difftodiff}]
    Let $\sigma=\sigma(1,4,t)$ and $\lambda=\lambda(1,4,t)$ be as in Lemma~\ref{lem:palanquintoalignment}. Let $\psi=\psi(t,\sigma+\lambda)$ be as in Theorem~\ref{banana} and let $o=o(\psi,d,3,t)$ be as in Lemma~\ref{ramsey2}. Our goal is to show that $\kappa=\kappa(d,t,w)=o+w$ satisfies \ref{thm:difftodiff}.
    
    Suppose not. Let $G$ be a (theta, prism, even wheel, $K_t$)-free graph, let $(a,x,y,\mathcal{W})$ be a $\kappa$-kaleidoscope in $G$, and let $z\in V(G)$ be $1$-mirrored by $(a,x,y,\mathcal{W})$. Let $\mathcal{W}'\subseteq \mathcal{W}$ be the set of all paths $W\in \mathcal{W}$ for which $z$ has at least $d$ neighbors in $W$. It follows that $|\mathcal{W}'|<w$, and so there exists $\mathcal{W}_0\subseteq \mathcal{W}$ with $|\mathcal{W}_0|=o$ such that for every $W\in \mathcal{W}_0$, $z$ has less than $d$ neighbors in $W$.

     For every $W\in \mathcal{W}_0$, traversing $W$ from $x$ to $y$, let $x_W$ be the last neighbor of $x$ in $W$, let $u^1_W, u^2_W$ be the first and the last neighbor of $z$ in $W$, respectively, and let $y_W$
be the first neighbor of $y$ in $W$. It follows that the vertices $x,x_W, u^1_W, u^2_W, y_W, y$ appear on $W$ in this order, and $u^1_W, u^2_W$ are the only two vertices among them which may be the same. Since $G$ is $(K_{3,3},K_t)$-free and from the choice of $o$, it follows that we can apply Lemma~\ref{ramsey2} to the sets $\{N_W(z):W\in \mathcal{W}_0\}$, and show that:

\sta{\label{st:danti} There exists $\mathcal{W}_1\subseteq \mathcal{W}_0$ with $|\mathcal{W}_1|=\psi$ such that for all distinct $W,W'\in \mathcal{W}_0$, $N_W(z)$ is anticomplete to $N_{W'}(z)$.}

Next, note that $\{z\dd u_W^1\dd W\dd x_W\dd x:W\in \mathcal{W}_1\}$ is a collection of $\psi$ pairwise internally disjoint paths in $G$ between non-adjacent vertices $z$ and $x$. Consequently, by the choice of $\psi$, we can apply Theorem~\ref{banana} to this collection, and deduce that there exist two disjoint subsets $\mathcal{W}_2$ and $\mathcal{W}_3$ of $\mathcal{W}_1$ with $|\mathcal{W}_2|=\sigma$ and $|\mathcal{W}_3|=\lambda$, such that:
\begin{itemize}
    \item $\{u^1_W:W\in \mathcal{W}_2\cup \mathcal{W}_3\}\cup \{x\}$ is a stable set in $G$ (though this is already guaranteed by \eqref{st:danti} and \ref{M3} as $z$ is $1$-mirrored by $(a,x,y,\mathcal{W})$); and
    \item for every $W\in \mathcal{W}_2$ and every $W'\in \mathcal{W}_3$, $u^1_{W}$ has a neighbor in the interior of $L_{W'}=u_{W'}^1\dd W'\dd x_{W'}\dd x$.
    \end{itemize}
    
   Let $S=\{u^1_W:W\in \mathcal{W}_2\}$ and let $\mathcal{L}=\{L^*_{W'}:W'\in \mathcal{W}_3\}$. Then $(z,S,\mathcal{L})$ is a $(\sigma,\lambda)$-palanquin in $G$. This, together with the choices of $\sigma$ and $\lambda$, allows for an application of Lemma~\ref{lem:palanquintoalignment}. We deduce that there exist $W_1,W_2,W_3,W_4\in \mathcal{W}_2$ and $W'\in \mathcal{W}_3$ such that the following hold.

\begin{itemize}
    \item $(\{u^1_{W_i}:i\in [4]\},L^*_{W'},x_{W'},\pi)$ is a $4$-alignment in $G$, where $\pi(u^1_{W_i})=i$ for all $i\in [4]$.
     \item For every $i\in [4]$, $u^1_{W_i}$ is $L_{W'}$-ugly.
 \end{itemize}

\sloppy Now, for each $i\in [4]$, traversing $L_{W'}^*$ starting at $x_{W'}$, let $v_i,v'_i$ be the first and the last neighbors of $u^1_{W_i}$ in $L_{W'}^*$, respectively; it follows that $\{v_i,v'_i\}\cap\{x_{W'},u_{W'}^1\}=\emptyset$ and  the vertices $x_{W'},v_1,v'_1,v_2,v'_2,v_3,v'_3,v_4,v'_4, u_{W'}^1, u_{W'}^2,y_{W'}$ appear on $W'$ in this order. Let $C=a\dd x\dd x_{W'}\dd W'\dd y_{W'}\dd y\dd a$. Then $C$ is a hole in $G$ and $u^1_{W_1}$ and $z$ are two adjacent vertices in $G\setminus C$, each with a neighbor in $C$. Also, $u^1_{W_1}$ is $L_{W'}$-ugly, and so $C$-ugly, and by \eqref{st:danti}, $u^1_{W_1}$ and $z$ have no common neighbor in $C$. This, combined with Theorem~\ref{thm:evenwheeltheta}, implies that $z$ is $C$-bad.  More precisely, $a$ and $z$ are not adjacent (though we do not use this), and $N_{C}(z)=N_{W'}(z)=\{u^1_{W'},u^2_{W'}\}$ is a  two-vertex clique in $G$. We further deduce that:
{
\begin{figure}[t!]
\centering

\includegraphics[scale=0.6]{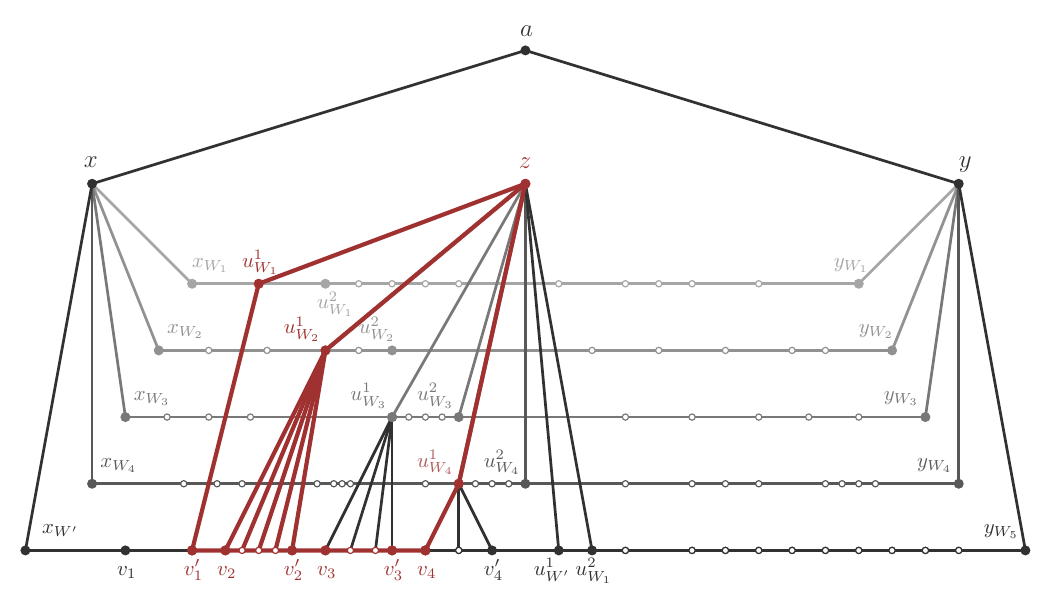}

\caption{Proof of \eqref{st:getneighborsjump}.}
\label{fig:diffractwheel}
\end{figure}
}
\sta{\label{st:getneighborsjump}For every $i\in \{2,3\}$, $u^1_{W_i}$ has a neighbor in the interior of $u^2_{W'}\dd W'\dd y_{W'}\dd y$.}

Suppose not. Then by \eqref{st:danti}, $u^1_{W_i}$ is anticomplete to $u^1_{W'}\dd u^2_{W'}\dd W'\dd y_{W'}$. Let $C'$ denote the hole $z\dd u^1_{W_1}\dd v'_1\dd W'\dd v_4\dd u^1_{W_4}\dd z$. Then we have $N_{C'}(u^1_{W_i})=N_{L_{W'}}(u^1_{W_i})\cup \{z\}=N_{C}(u^1_{W_i})\cup \{z\}$. Also, $u^1_{W_i}$ is $C$-ugly, as it is $L_{W'}$-ugly. This, along with the fact that $C\cup \{u^1_{W_i}\}$ is not a theta in $G$ and $(C,u^1_{W_i})$ is not an even wheel in $G$, implies that $|N_C(u^1_{W_i})|$ is an odd integer which is at least three. But then $(C',u^1_{W_i})$ is an even wheel in $G$, a contradiction (see Figure~\ref{fig:diffractwheel}). This proves \eqref{st:getneighborsjump}.

\medskip

To finish the proof, note that by \eqref{st:getneighborsjump}, there exists a path $P$ in $G$ from $u^1_{W_2}$ to $u^1_{W_3}$ such that $P^*$ is contained in the interior of $u^2_{W'}\dd W'\dd y_{W'}\dd y$. But now there is a theta in $G$ with ends $u^1_{W_2}, u^1_{W_3}$ and paths $u^1_{W_2}\dd z\dd u^1_{W_3}, u^1_{W_2}\dd v'_2\dd W'\dd v_3\dd u^1_{W_3}$ and $P$, a contradiction (see Figure~\ref{fig:diffracttheta}). This completes the proof of Theorem~\ref{thm:difftodiff}.
\end{proof}
{
\begin{figure}[t!]
\centering

\includegraphics[scale=0.6]{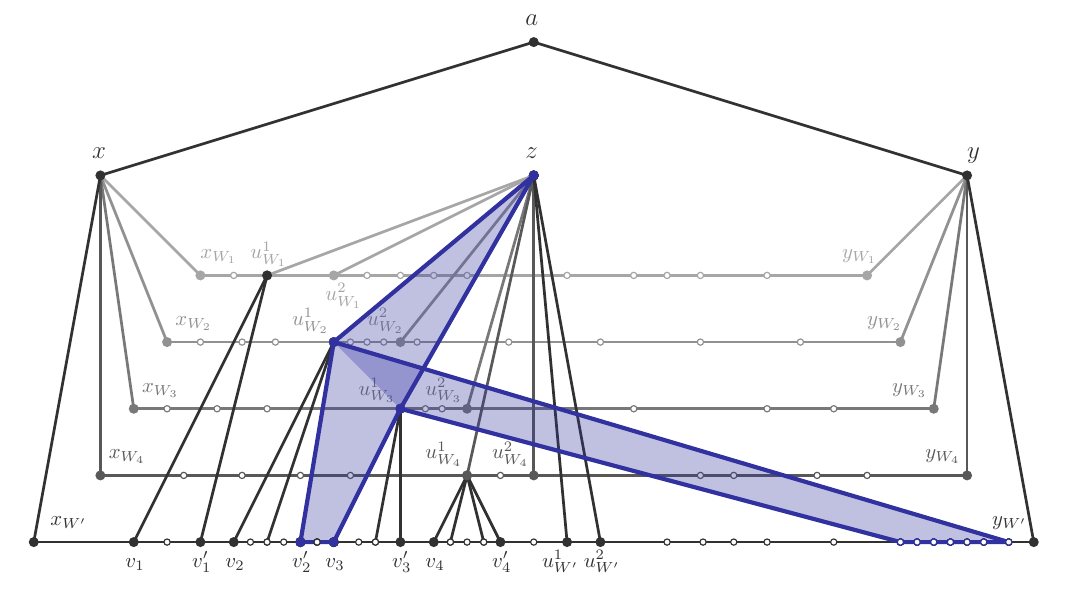}

\caption{Proof of Theorem~\ref{thm:difftodiff}.}
\label{fig:diffracttheta}
\end{figure}
}
\section{Blurry $2$-trees}\label{sec:kaleidoscope}

Here we develop the inductive procedure discussed in Subsection~\ref{subs:outline} for growing a $2$-tree while maintaining its vertex set, all along, mirrored enough by a kaleidoscope of proportionate size. In particular, this process starts with a pair of adjacent vertices and then repeatedly adds carefully chosen common neighbors. Let us first ensure that a valid choice of the two initial vertices is available under the right circumstances:

\begin{theorem}\label{thm:kaleidoscopeexists}
    For all integers $d,t,w\geq 1$, there exists $\zeta=\zeta(d,t,w)\geq 1$ with the following property. Let $G$ be a (theta, prism, $K_t$)-free graph, let $a,b\in V(G)$ be distinct and non-adjacent and let $\mathcal{P}$ be a collection of pairwise internally disjoint paths in $G$ from $a$ to $b$ with $|\mathcal{P}|\geq \zeta$.
    Then there exists a $w$-kaleidoscope $(a,x,y,\mathcal{W})$ in $G$ as well as a clique $Z_0$ in $G$ with $|Z_0|=2$, such that:
    \begin{enumerate}[{\rm (a)}]
        \item \label{thm:kaleidoscopeexists_a} $Z_0$ is $d$-mirrored by $(a,x,y,\mathcal{W})$; and
        \item \label{thm:kaleidoscopeexists_a} some vertex in $Z_0$ is adjacent to $a$.
    \end{enumerate}
    \end{theorem}

    \begin{proof}
        Let $\psi(\cdot,\cdot)$ be as in Theorem~\ref{banana} and let $\kappa=\kappa(\cdot,\cdot,\cdot)$ be as in Theorem~\ref{thm:difftodiff}. Define $\kappa_1=\kappa(d,t,w)$,
$\psi_1=\psi(t,t^3+\kappa_1+1)$ and 
$\kappa_2=\kappa(d,t,\psi_1)$.
Let $\sigma=\sigma(\kappa_2,3,t)$ and $\lambda=\lambda(\kappa_2,3,t)$ be as in 
        be as in Lemma~\ref{lem:palanquintoalignment}. Define
$\zeta=\zeta(t,w)=\psi(t,\sigma+\lambda)$. We prove that this value of $\zeta$
satisfies \ref{thm:kaleidoscopeexists}. Let $G$ be a (theta, prism, $K_t$)-free graph, let $a,b\in V(G)$ be distinct and non-adjacent and let $\mathcal{P}$ be a collection of pairwise internally disjoint paths in $G$ from $a$ to $b$ with $|\mathcal{P}|\geq \zeta=\psi(t,\sigma+\lambda)$. For each $P\in \mathcal{P}$, let $x_{P}$ be the neighbor of $a$ in $P$ (so $x_P\neq b$). From Theorem~\ref{banana} applied to $a,b$ and $\mathcal{P}$,  we deduce that there exist disjoint subsets $\mathcal{Q}$ and $\mathcal{R}$ of  $\mathcal{P}$ with $|\mathcal{Q}|=\sigma$ and $|\mathcal{R}|=\lambda$, such that:
    \begin{itemize}
    \item $\{x_{P}:P\in \mathcal{Q}\cup \mathcal{R}\}\cup  \{b\}$ is a stable set in $G$; and
    \item for every $Q\in \mathcal{Q}$ and every $R\in \mathcal{R}$, $x_{Q}$ has a neighbor in $R^*\setminus \{x_{R}\}$.
    \end{itemize}
For every $R\in \mathcal{R}$, let $L_R=R^*\setminus \{x_{R}\}$, and let $x'_{L_{R}}$ be the unique neighbor of $x_{R}$ in $L_{R}$. Then $x'_{L_R}$ is an end of $L_R$, and the vertices $\{x'_{L_R}:R\in \mathcal{R}\}$ are pairwise distinct.

We define $S=\{x_{Q}:Q\in \mathcal{Q}\}$ and  $\mathcal{L}=\{L_R:R\in \mathcal{L}_R\}$. Note that $a$ is complete to the stable set $S\subseteq \{x_P:P\in \mathcal{P}\}$ and anticomplete to $P^*\setminus x_P$ for every $P\in \mathcal{P}$. Therefore, the triple $(a,S,\mathcal{L})$ is a $(\sigma,\lambda)$-palanquin in $G$. For every path $L\in \mathcal{L}$, fix the end $x'_L$ of $L$, chosen as above.

Since $G$ is (theta, $K_t$)-free and due to the choices of $\sigma$ and $\lambda$, we can apply Lemma~\ref{lem:palanquintoalignment} to $(a,S,\mathcal{L})$ together with $\{x'_L:L\in \mathcal{L}\}$, and obtain a stable set $S'\subseteq S$ with $|S'|=3$, $\mathcal{L}'\subseteq \mathcal{L}$ with $|\mathcal{L}'|=\kappa$, and a bijection $\pi:S'\rightarrow [3]$, such that for every $L\in \mathcal{L}'$, $(S',L,x'_L, \pi)$ is a $3$-alignment in $G$.

Let us write $x=\pi(1),z_1=\pi(2)$ and $y=\pi(3)$. For every $L\in \mathcal{L}'$, traversing $L$ starting at $x'_L$, let $u_L$ be the last neighbor of $x$ in $L$, let $z_L$ be the last neighbor of $z_1$ in $L$ and let $v_L$ be the first neighbor of $y$ in $L$. Let $W_L=x\dd u_L\dd L\dd v_L\dd y$. Then $W_L$ is a path in $G$ from $x$ to $y$ and we have $z_L\in W_L\setminus (N_{W_L}[x]\cup N_{W_L}[y])$. In particular:

\sta{\label{st:getz1} For every $L\in \mathcal{L}'$, $z_1$ is anticomplete to $N_{W_L}[x]\cup N_{W_L}[y]$ and $z_1$ has a neighbor in $W_L$ (namely $z_L$).}

From \eqref{st:getz1}, it follows that $(a,x,y,\{W_L:L\in \mathcal{L}'\})$ is a $\kappa_2$-kaleidoscope in $G$ by which $z_1$ is $1$-mirrored. This, along with the choice of $\kappa_2$ and Theorem~\ref{thm:difftodiff}, implies that there exists $\mathcal{L}_1\subseteq \mathcal{L}'$ with $|\mathcal{L}_1|=\psi_1$ such that $z_1$ is $d$-mirrored by the $\psi_1$-kaleidoscope $(a,x,y,\{W_L:L\in \mathcal{L}_1\})$.

Next, for every path $L\in \mathcal{L}_1$, let $P_L=z_1\dd z_L\dd L\dd v_L\dd y$. Then $\mathcal{P}'=\{P_L:L\in \mathcal{L}_1\}$ is a collection of $\psi_1$ pairwise internally disjoint paths in $G$ between the two non-adjacent vertices $z_1$ and $y$. By Theorem~\ref{banana}, this time applied to $z_1,y$ and $\mathcal{P}'$, there exist $L_0,L_1,\ldots, L_{t^3+\kappa}\in \mathcal{L}'$ such that:
    \begin{itemize}
    \item $\{z_{L_0},z_{L_1},\ldots, z_{L_{t^3+\kappa_1}},y\}$ is a stable set in $G$; and
    \item for all $j\in [t^3+\kappa_1]$, $z_{L_0}$ has a neighbor in $P_{L_j}^*\setminus \{z_{L_j}\}$.
    \end{itemize}
Let $z_2=z_{L_0}$. Then $z_2$ is anticomplete to $\{a,x,y\}$. In addition, we claim that:

\sta{\label{st:antitoends} The vertex $z_2$ is anticomplete to $\{u_{L_j}:j\in [t^3+\kappa_1]\}$. Moreover, there exists $I\subseteq [t^3+\kappa_1]$ with $|I|=\kappa_1$ for which $z_2$ is anticomplete to $\{v_{L_j}:j\in I\}$. Consequently, for every $j\in I$, $z_2$ is anticomplete to $\{a\}\cup N_{W_{L_j}}[x]\cup N_{W_{L_j}}[y]$, and $z_2$ has a neighbor in $W_{L_j}$.}

To see the first assertion, note that for all $j\in [t^3+\kappa_1]$, there is a path $Q_{j}$ in $G$ from $z_2$ to $y$ with $Q_{j}^*\subseteq P_{L_j}^*\setminus \{z_{L_j}\}$; in particular, $u_{L_j}$ is anticomplete to $Q^*_{j}$. Therefore, if $z_2$ is adjacent to $u_{L_j}$ for some $j\in [t^3+\kappa_1]$, then there is a theta in $G$ with ends $a,z_2$ and paths $a\dd x\dd u_{L_j}\dd z_2, a\dd z_1\dd z_2$ and $a\dd y\dd Q_{j}\dd z_2$, a contradiction. To prove the second assertion, suppose for a contradiction that there exists $I'\subseteq [t^3+\kappa_1]$ with $|I'|=t^3$ such that $z_2$ is complete to $V'=\{v_{L_j}:j\in I'\}$. Since $G$ is $K_t$-free, it follows from Theorem~\ref{classicalramsey} applied to $G[V']$ that there exists $\{v,v',v''\}\subseteq V'$ which is a stable set in $G$. But this yields a theta in $G$ with ends $z_2,y$ and paths $z_2\dd v\dd y, z_2\dd v'\dd y$ and $z_{2}\dd v''\dd y$, a contradiction.  This proves \eqref{st:antitoends}.
\medskip

Let $I$ be as in \eqref{st:antitoends}. It follows that $(a,x, y,\{W_{L_j}:j\in I\})$ is a $\kappa_1$-kaleidoscope in $G$ by which $z_1$ is $d$-mirrored and $z_{2}$ is $1$-mirrored. From the choice of $\kappa_1$ and Theorem~\ref{thm:difftodiff} applied to $(a,x, y,\{W_{L_j}:j\in I\})$ and $z_2$, we conclude that there exists $\mathcal{W}\subseteq \{W_{L_j}:j\in I\}$ with $|\mathcal{W}|=w$ such that $Z_0=\{z_1,z_2\}$ is $d$-mirrored by the $w$-kaleidoscope $(a,x,y,\mathcal{W})$. Also, $a$ is adjacent to $z_1\in Z_0$ and non-adjacent to $z_2\in Z_0$. This completes the proof of Theorem~\ref{thm:kaleidoscopeexists}.
\end{proof}

 Incidentally, there is an immediate corollary of Theorem~\ref{thm:kaleidoscopeexists} which may be of independent interest. Let $G$ be a graph and let $d\geq 1$ be an integer. We say vertex $v\in V(G)$ is \textit{$d$-substantial} if there is a hole $C$ in $G\setminus \{v\}$ such that $v$ has at least $d+1$ neighbors in $C$ and $C\setminus N_C(v)$ is not connected.
\begin{corollary}\label{cor:hubnbrfortw10}
    For all integers $d,t\geq 1$, there exists an integer $k=k(d,t)$ with the following property. Let $G$ be a (theta, prism, even wheel, $K_t$)-free graph and let $a,b\in V(G)$ be distinct and non-adjacent. Assume that no vertex in $N_G(a)$ is $d$-substantial in $G$. Then there do not exist $k$ pairwise internally disjoint paths in $G$ from $a$ to $b$.
\end{corollary}
\begin{proof}
We show that $k(d,t)=\zeta(d,t,1)$ satisfies \ref{cor:hubnbrfortw10}, where $\zeta(\cdot,\cdot,\cdot)$ comes from Theorem~\ref{thm:kaleidoscopeexists}. Suppose for a contradiction that there is a collection $\mathcal{P}$ of pairwise internally disjoint paths in $G$ from $a$ to $b$ with $|\mathcal{P}|\geq k$. By Theorem~\ref{thm:kaleidoscopeexists}, there exists a $1$-kaleidoscope $(a,x,y,\{W\})$ in $G$ as well as a clique $Z_0$ in $G$ with $|Z_0|=2$ such that $Z_0$ is $d$-mirrored by $(a,x,y,\{W\})$, and some vertex $z_1\in Z_0$ is adjacent to $a$. Let $C=a\dd x\dd W\dd y\dd a$. Then $C$ is a hole in $G$ and $z_1$ has at least $d+1$ neighbors in $C$. Also, the vertices $x,y$ belong to distinct components of $C\setminus N_C(z_1)$. But now $z_1$ is neighbor of $a$ in $G$ which is $d$-substantial, a contradiction.
\end{proof}

Back to our main theme, the $2$-trees we are about to obtain are in fact subgraphs of their host graphs, falling short of being induced by only an (annoying) notch: for a graph $G$ and a $2$-tree $\nabla$ with $|V(\nabla)|=h\geq 2$,  by a \textit{blurry copy of $\nabla$ in $G$} we mean an induced subgraph $Z$ of $G$ which in turn contains a spanning subgraph $Y$ with the following specifications.

\begin{enumerate}[(B1), leftmargin=15mm, rightmargin=7mm]
\item\label{B1} $Y$ is a $2$-tree isomorphic to $\nabla$.
\item\label{B2} Let $i,j\in [|Y|]=[|Z|]=[|V(\nabla)|]$ with $i<j$ for which $\varpi_Y(i)$ and $\varpi_Y(j)$ are adjacent in $G$ (and so in $Z$) but not in $Y$. Then $\varpi_Y(j)$ is adjacent in $G$ to both forward neighbors of $\varpi_Y(i)$ in $Y$. 
\end{enumerate}

In particular, it follows that:

\begin{observation}\label{obs:blurry}
    Let $G$ be a graph and let $\nabla$ be a $2$-tree such that $G$ there is a blurry copy of $\nabla$ in $G$. Then $G$ has a subgraph isomorphic to $\nabla$. Moreover, if $G$ is $K_4$-free, then $G$ has an induced subgraph isomorphic to $\nabla$.
\end{observation}
Next we demonstrate, through the following lemma, the inductive step of our $2$-tree growing process. Given a $2$-tree $\nabla$ on $h+2$ vertices for $h\geq 1$ and an integer $i\in [h]$, we denote by $\nabla/i$ the $2$-tree $\nabla\setminus (\varpi_{\nabla}[i])$ on $h-i+2$ vertices where $\varpi_{\nabla/i}(j)=\varpi_{\nabla}(i+j)$ for all $j\in [h-i+2]$. It is also useful to define $\nabla/0=\nabla$.

\begin{figure}[t!]
\centering

\includegraphics[scale=0.6]{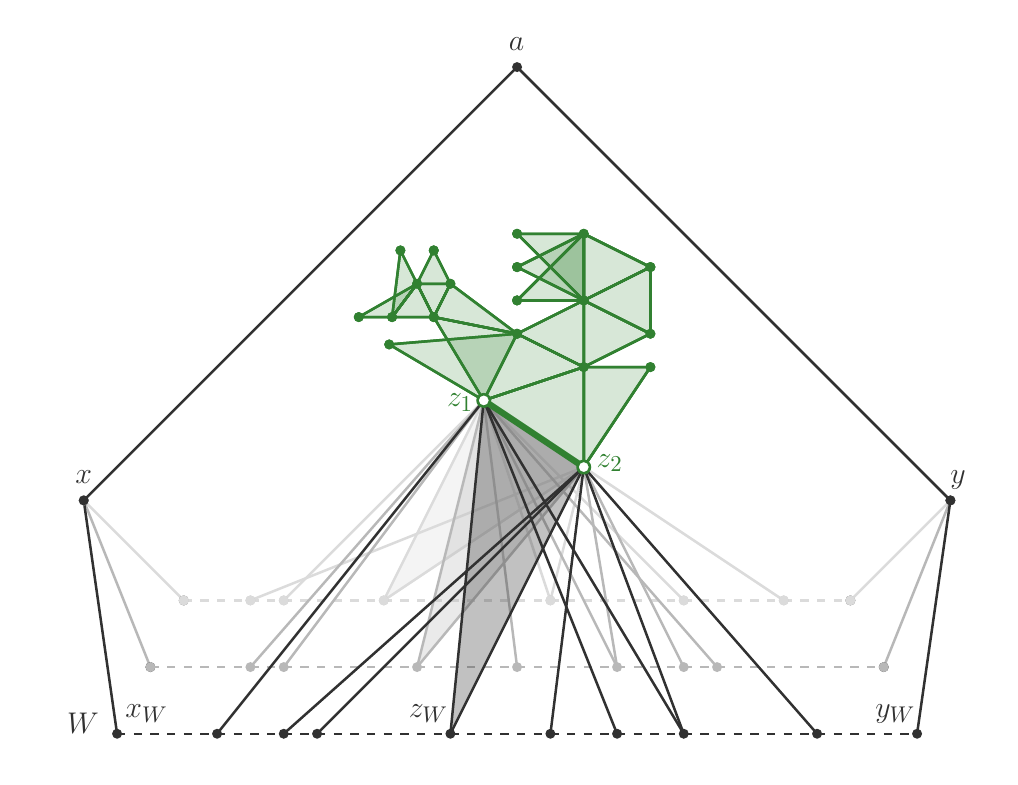}

\caption{Proof of Lemma~\ref{lem:mainblurryinduct} (dashed lines represent paths of arbitrary lengths).}
\label{fig:lhbij}
\end{figure}

\begin{lemma}\label{lem:mainblurryinduct}
    For all integers $h,t,w\geq 1$, there exists an integer $\xi(h,t,w)\geq 1$ with the following property. Let $G\in \mathcal{E}_t$ be a graph and let $\nabla$ be a $2$-tree on $h+2$ vertices. Assume that there is a blurry copy $Z'$ of $\nabla/1$ in $G$ which is $3$-mirrored by a $\xi$-kaleidoscope in $G$. Then there is a blurry copy $Z$ of $\nabla$ in $G$ which is $3$-mirrored by a $w$-kaleidoscope in $G$.
\end{lemma}
   
\begin{proof}
    Let $\psi=\psi(t,2)$ be as in Theorem~\ref{banana}. Let $\kappa=\kappa(3,t,w)$ be as in Theorem~\ref{thm:difftodiff}. Let $o=o(\cdot,\cdot,\cdot,\cdot)$ be as in Lemma~\ref{ramsey2}. We claim that
$\xi=\xi(h,t,w)=o(3h+2\psi\kappa,3,3,t)$
satisfies \ref{lem:mainblurryinduct}. Let $G\in \mathcal{E}_t$ be a graph and let $\nabla$ be $2$-tree on $h+2$ vertices. Let $Z'$ be a blurry copy of $\nabla/1$ in $G$ and let $(a,x,y,\mathcal{W}')$ be a $\xi$-kaleidoscope in $G$ by which $Z'$ is $3$-mirrored. Let $Y'$ be the spanning subgraph of $Z'$ such that $Y',Z'$ and $\nabla/1$ satisfy \ref{B1} and \ref{B2}. In particular, there is an isomorphism $f:V(\nabla/1)\rightarrow V(Y')$ between the $2$-trees $\nabla/1$ and $Y'$. Let $u_1,u_2$ be the two forward neighbors of $\varpi_{\nabla}(1)$ in $\nabla$ and let $z_i=f(u_i)$ for $i\in \{1,2\}$. Then we have $z_1,z_2\in V(Y')=Z'$. Since $Z'$ is a blurry copy of $\nabla/1$ which is $3$-mirrored by $(a,x,y,\mathcal{W}')$, and from the definition of a blurry copy and a mirrored set, it follows immediately that:

\sta{\label{st:plan} Let $z\in G\setminus Z'$ and $\mathcal{W}\subseteq \mathcal{W}'$ such that, $a$ is not adjacent to $z$,  we have $\{z_1,z_2\}\subseteq N_{Z'}(z)\subseteq N_{Z'}[z_1]
    \cap N_{Z'}[z_2]$, and we have $|\mathcal{W}|=w$ and $z$ is $3$-mirrored by the $w$-kaleidoscope $(a,x,y,\mathcal{W})$ in $G$. Then $Z=Z'\cup \{z\}$ is a blurry copy of $\nabla$ in $G$ which is $3$-mirrored by the $w$-kaleidoscope $(a,x,y,\mathcal{W})$ in $G$.}

Therefore, in order to prove Lemma~\ref{lem:mainblurryinduct}, it suffices to argue the existence of a vertex $z\in G\setminus Z'$ as well as a subset $\mathcal{W}$ of $\mathcal{W}'$ for which the three bullets points of \eqref{st:plan} hold. We devote the rest of the proof to this goal.

For every $W\in \mathcal{W}'$, let $C_W=a\dd x\dd W\dd y\dd a$. Then $C_W$ is a hole in $G$. Also, since $Z'$ is $3$-mirrored by $(a,x,y,\mathcal{W}')$, it follows from \ref{M2} and \ref{M3} that $a$ is not a common neighbor of $z_1$ and $z_2$, and that both $z_1$ and $z_2$ are $C_W$-ugly. This, combined with Theorem~\ref{thm:evenwheeltheta}, implies that:

\sta{\label{st:getcommon} For every $W\in \mathcal{W}'$, the vertices $z_1$ and $z_2$ have a common neighbor in $W$.}

Consequently, for every $W\in \mathcal{W}'$, traversing $W$ from $x$ to $y$, we may choose $z_W$ to be the first common neighbor of $z_1$ and $z_2$ in $W$. Note also that, by \ref{M3}, $\{z_1,z_2\}$ is anticomplete to $N_{W}[x]\cup N_{W}[y]$, which in turn shows that $z_W\in W\setminus (N_{W}[x]\cup N_{W}[y])$. For every $W\in \mathcal{W}'$, let $x_W,y_W$ be the neighbors of $x$ and $y$ in $W$, respectively (see Figure~\ref{fig:lhbij}). We now deduce that:

\sta{\label{st:antiinduct} There exists $\mathcal{W}_0\subseteq \mathcal{W}'$ with $|\mathcal{W}_0|=2\psi\kappa$ such that:
\begin{itemize}
    \item for all distinct $W,W'\in \mathcal{W}_0$, the sets $\{x_W,z_W,y_W\}$ and $\{x_{W'},z_{W'},y_{W'}\}$ are anticomplete in $G$; and
    \item for every $W\in \mathcal{W}_0$, we have $\{z_1,z_2\}\subseteq N_{Z'}(z_W)\subseteq N_{Z'}[z_1]\cap N_{Z'}[z_2]$.
\end{itemize}}

To see this, note that since $G$ is $(K_{3,3},K_t)$-free and from choice of $|\mathcal{W}'|=\xi$, we can apply Lemma~\ref{ramsey2} to the sets $\{\{x_W,z_W,y_W\}:W\in \mathcal{W}'\}$, and obtain a subset $\mathcal{W}'_0$ of $\mathcal{W}'$ with $|\mathcal{W}'_0|=3h+2\psi\kappa$ such that for all distinct $W,W'\in \mathcal{W}'_0$, the sets $\{x_W,z_W,y_W\}$ and $\{x_{W'},z_{W'},y_{W'}\}$ are anticomplete in $G$. It remains to show that there exists a subset $\mathcal{W}_0$ of $\mathcal{W}_0'$ with $|\mathcal{W}_0|=2\psi\kappa$ such that for every $W\in \mathcal{W}_0$, we have $\{z_1,z_2\}\subseteq N_{Z'}(z_W)\subseteq N_{Z'}[z_1]\cap N_{Z'}[z_2]$. Suppose not. Since $z_W$ is a common neighbor of $z_1$ and $z_2$ for all $W\in \mathcal{W}'$, it follows that there exists $\mathcal{W}''_0\subseteq \mathcal{W}'_0$ with $|\mathcal{W}''_0|=3h$ such that for every $W\in \mathcal{W}''_0$, $z_W$ has a neighbor $z'_W\in Z'\setminus \{z_1,z_2\}$ which is adjacent to at most one of $z_1$ and $z_2$. From this and the fact that $|Z'\setminus \{z_1,z_2\}|<|V(\nabla)|-2=h$, we deduce that there are three distinct paths $W_1,W_2,W_3\in \mathcal{W}''_0$ such that for some vertex $z'\in Z'\setminus \{z_1,z_2\}$, we have $z'_{W_1}=z'_{W_2}=z'_{W_3}=z'$. Recall also that $z_{W_1},z_{W_2},z_{W_3}$ are pairwise non-adjacent because $W_1,W_2,W_3\in \mathcal{W}''_0\subseteq \mathcal{W}'_0$. But now for some $i\in \{1,2\}$, there is a theta in $G$ with ends $z_i,z'$ and paths $z_i\dd z_{W_1}\dd z', z_i\dd z_{W_2}\dd z'$ and $z_i\dd z_{W_3}\dd z'$, a contradiction. This proves \eqref{st:antiinduct}.
\medskip

Let $\mathcal{W}_0$ be as in \eqref{st:antiinduct}. The following captures the bulk of the difficulty in this proof, also involving our application of Theorem~\ref{thm:<3}. Intuitively, the motivation is to apply Theorem~\ref{banana} to the ``paths in $\mathcal{W}_0$'' from $z_1$ to $x$. But we cannot; those paths are not induced as $z_1$ may have neighbours in them. So we pass to the contraption precisely to surmount this complication.

\sta{\label{st:applymiracle} There exist $W_0\in \mathcal{W}_0$ and $\mathcal{W}_1\subseteq \mathcal{W}_0\setminus \{W_0\}$ with $|\mathcal{W}_1|=\kappa$, such that for every $W\in \mathcal{W}_1$, $z_{W_0}$ has a neighbor in $W^*$.}

Let $D$ be the digraph with vertex set $\mathcal{W}_0$ where for distinct paths $W_1,W_2\in \mathcal{W}_0$, $(W_1,W_2)$ is an arc in $D$ if and only if $z_{W_1}$ has a neighbor in $W_2^*$. Note that in order to prove \eqref{st:applymiracle}, it is enough to show that $D$ has a vertex of out-degree at least $\kappa$. Suppose not. Then every vertex in $D$ has out-degree less than $\kappa$, which in turn implies that every vertex in every ``subdigraph'' of $D$ has out-degree less than $\kappa$. It follows that every ``subdigraph'' of $D$ has a vertex of in-degree less than $\kappa$.  Let $D^{\natural}$ be the underlying undirected graph of $D$ (which may have pairs of parallel edges). Then every subgraph of $D^{\natural}$ has a vertex of degree less than $2\kappa$, and so $D^{\natural}$ has chromatic number at most $2\kappa$. Consequently, there exists a stable set $\mathcal{W}'_1\subseteq \mathcal{W}_0=V(D^{\natural})$ in $D^{\natural}$ with $|\mathcal{W}'_1|=\lceil |\mathcal{W}_0|/2\kappa\rceil=\psi$. From the definition of $D$, it follows that for all distinct $W_1,W_2\in \mathcal{W}'_1$, $z_{W_1}$ is anticomplete to $W_2^*$. Let $G_1=G[(\bigcup_{W\in \mathcal{W}'_1}V(z_W\dd W\dd x))\cup \{z_1,z_2\}]$. Then we have $G_1\in \mathcal{E}$, $z_1,z_2\in V(G_1)$ are distinct and adjacent, and $N_{G_1}(z_1)\cap N_{G_1}(z_2)=\{z_W:W\in \mathcal{W}'_1\}$ is a stable set of vertices of degree three in $G_1$. Let $G'_1$ be the $z_1z_2$-contraption of $G_1$. We are now prepared to apply Theorem~\ref{thm:<3} and deduce that $G'_1\in \mathcal{E}$. Let $z\in V(G'_1)$ be as in the definition of $G'_1\in \mathcal{E}$. Then we have $V(G'_1)=(\bigcup_{W\in \mathcal{W}'_1}V(z_W\dd W\dd x ))\cup \{z\}$ and $N_{G'_1}(z)=N_{G_1}(z_1)\cap N_{G_1}(z_2)=\{z_W:W\in \mathcal{W}'_1\}$ is a stable set of vertices of degree two in $G'_1$. Also, $\{z\dd z_W\dd W\dd x:W\in \mathcal{W}_1'\}$ is a collection of $\psi$ pairwise internally disjoint paths in $G'_1$ between non-adjacent vertices $z$ and $x$.
Hence, by the choice of $\psi$, we can apply Theorem~\ref{banana} to this collection, and obtain $W_1,W_2\in \mathcal{W}_1'$ such that $z_{W_1}$ has a neighbor in the interior of $z_{W_2}\dd W_2\dd x$ in $G'_1$. But then $z_{W_1}$ has degree at least three in $G'_1$, a contradiction. This proves \eqref{st:applymiracle}.

\medskip

To finish the proof, let $W_0$ and $\mathcal{W}_1$ be as in \eqref{st:applymiracle}, and write $z=z_{W_0}$. Then $z\in G\setminus Z'$ and $a$ is not adjacent to $z$ because $z\in W_0^*$, which implies that $a$ satisfies the first bullet of \eqref{st:plan}. Also, the second bullet of \eqref{st:antiinduct} implies that $z$ satisfies the first second of \eqref{st:plan}. In addition, from the first bullet of \eqref{st:antiinduct} combined with \eqref{st:applymiracle} and the fact that $z\in W_0\setminus (N_{W_0}[x]\cup N_{W_0}[y])$, it follows that $z$ is $1$-mirrored by the $\kappa$-kaleidoscope $(a,x,y,\mathcal{W}_1)$. Due to the choice of $\kappa$, we may apply Theorem~\ref{thm:difftodiff} and deduce that there exists $\mathcal{W}\subseteq \mathcal{W}_1\subseteq \mathcal{W}_0\subseteq \mathcal{W}'$ with $|\mathcal{W}|=w$ such that $z$ is $3$-mirrored by the $w$-kaleidoscope $(a,x,y,\mathcal{W})$. Hence, $z$ and $\mathcal{W}$ satisfy the third bullet of \eqref{st:plan}. This completes the proof of Lemma~\ref{lem:mainblurryinduct}.
    \end{proof}

We now use Lemma~\ref{lem:mainblurryinduct} to prove the main result of this section:

    \begin{theorem}\label{thm:mainblurry}
      For all integers $h\geq 0$ and $t,w\geq 1$, there is an integer $\Xi(h,t,w)\geq 1$ with the following property. Let $G\in \mathcal{E}_t$ be a graph and let $\nabla$ be a $2$-tree on $h+2$ vertices. Assume that there is a clique $Z_0$ in $G$ with $|Z_0|=2$ which is $3$-mirrored by a $\Xi$-kaleidoscope in $G$. Then there is a blurry copy $Z$ of $\nabla$ in $G$ which is $3$-mirrored by a $w$-kaleidoscope in $G$.
    \end{theorem}
    \begin{proof}
        We begin with defining a sequence $\{\Xi_i\}_{i=0}^h$ using a backward recursion. Let  $\Xi_h=w$. For every $0\leq i<h$, let $\Xi_i=\xi(i+1,t,\Xi_{i+1})$, where $\xi(\cdot,\cdot,\cdot)$ is as in Lemma~\ref{lem:mainblurryinduct}. We claim that $\Xi=\Xi(h,t,w)=\Xi_0$ satisfies \ref{thm:mainblurry}. In fact, we prove a slightly stronger statement which is tailored to our inductive argument:

    \sta{\label{st:inductwithin} Let $i\in \{0,\ldots, h\}$. Assume that there is a clique $Z_0$ in $G$ with $|Z_0|=2$ which is $3$-mirrored by a $\Xi_0$-kaleidoscope in $G$. Then there is a blurry copy $Z_i$ of $\nabla/(h-i)$ in $G$ which is $3$-mirrored by a $\Xi_i$-kaleidoscope in $G$.}

   We induct on $i$. The case $i=0$ is trivial as $\nabla/h$ is a $2$-vertex complete graph. Suppose $h\geq 1$ and $i\in [h]$. Assume that there is a clique $Z_0$ in $G$ with $|Z_0|=2$ which is $3$-mirrored by a $\Xi_0$-kaleidoscope in $G$. Then $\nabla/(h-i)$ is a $2$-tree on $i+2$ vertices, and by the induction hypothesis, there is a blurry copy $Z_{i-1}$ of $\nabla/(h-i+1)=(\nabla/(h-i))/1$ in $G$ which is $3$-mirrored by a $\Xi_{i-1}$-kaleidoscope in $G$, where $\Xi_{i-1}=\xi(i,t,\Xi_i)$. Hence, it follows from Lemma~\ref{lem:mainblurryinduct} that there is a blurry copy $Z_{i}$ of $\nabla/h-i$ in $G$ which is $3$-mirrored by a $\Xi_{i}$-kaleidoscope in $G$. This proves \eqref{st:inductwithin}.

   \medskip

   Now the result follows from \eqref{st:inductwithin} for $i=h$. This completes the proof of Theorem~\ref{thm:mainblurry}.
    \end{proof}

\section{The coda}\label{sec:end}
With Theorems~\ref{thm:kaleidoscopeexists} and \ref{thm:mainblurry} in our arsenal, we are ready to prove Theorems~\ref{thm:main2-tree} and \ref{thm:tree+}. In fact, both results follow directly from the one below:
\begin{theorem}\label{thm:bigtwblurry}
    For all integers $h\geq 0$ and $t\geq 1$, there exists an integer $\Omega=\Omega(h,t)$ such that for every $2$-tree $\nabla$ on $h$ vertices and every graph $G\in \mathcal{E}_t$ with $\tw(G)>\Omega$, there is a blurry copy of $\nabla$ in $G$.
\end{theorem}
\begin{proof}
    Let $\Xi=\Xi(h,t,1)$ be as in Theorem~\ref{thm:mainblurry} and let $\zeta=\zeta(3,t,\Xi)$ be as in Theorem~\ref{thm:kaleidoscopeexists}. We show that $\Omega=\Omega(h,t)=\beta(\max\{\zeta,t\},t)$ satisfies \ref{thm:bigtwblurry}, where $\beta(\cdot,\cdot)$ is as in Corollary~\ref{cor:noblocksmalltw_Ct}. Let $\nabla$ be a $2$-tree on $h$-vertices and let $G\in \mathcal{E}_t$ be a graph with $\tw(G)>\Omega$. From Corollary~\ref{cor:noblocksmalltw_Ct},  it follows that $G$ has a strong $\max\{\zeta,t\}$-block. Since $G$  is $K_t$-free, it follows that there are non-adjacent vertices $a,b\in V(G)$ for which there exists a collection $\mathcal{P}$ of $\zeta$ pairwise internally disjoint paths in $G$ from $a$ to $b$. By Theorem~\ref{thm:kaleidoscopeexists},  there exists a clique $Z_0$ in $G$ with $|Z_0|=2$ which is $3$-mirrored by a $\Xi$-kaleidoscope in $G$. This, along with Theorem~\ref{thm:mainblurry}, implies that there is a blurry copy of $\nabla$ in $G$, as desired.
\end{proof}
Theorem~\ref{thm:main2-tree} is now immediate:
\setcounter{section}{1}
\setcounter{theorem}{11}
\begin{theorem}
    For every $2$-tree $\nabla$, there exists an integer $\Upsilon=\Upsilon(\nabla)\geq 1$ such that every graph $G\in \mathcal{E}_4$ with $\tw(G)>\Upsilon$ contains $\nabla$.
\end{theorem}
\begin{proof}
    Let $\Upsilon(\nabla)=\Omega(|V(\nabla)|,4)$. Then the result follows from Theorem~\ref{thm:bigtwblurry} and Observation~\ref{obs:blurry}.
\end{proof}
To prove Theorem~\ref{thm:tree+}, we also need a result from \cite{KP}. For integers $d,r\geq 0$, let $T_d^r$ denote the rooted tree in which every leaf is at distance $r$ from the root, the root has degree $d$, and every vertex that is neither a leaf nor the root has degree $d+1$.
\setcounter{section}{6}
\setcounter{theorem}{1}
\begin{theorem}[Kierstead and Penrice \cite{KP}]\label{KP}
    For all integers $d, r\geq 0$ and $s, t\geq 1$, there exists an integer $f= f(d,r,s,t)\geq 1$ such that if a graph $G$ contains $T_f^f$ as a subgraph, then $G$ contains one of $K_{s,s}$, $K_t$ or $T_d^r$ as an induced subgraph.
\end{theorem} 
Finally, we prove Theorem~\ref{thm:tree+}:
\setcounter{section}{1}
\setcounter{theorem}{12}
\begin{theorem}
    For every integer $t\geq 1$ and every tree $T$, there exists an integer $\Gamma=\Gamma(t,T)\geq 1$ such that every graph $G\in \mathcal{E}_t$ with $\tw(G)>\Gamma$ contains $\cone(T)$.
\end{theorem}
\begin{proof}
    Let $d$ and $r$ be the maximum degree and the radius of $T$, respectively. It follows that $T_d^r$ contains $T$ as an induced subgraph. Let $f = f (d, r, 3, t)$ be as in Theorem~\ref{KP} and let $T^+=\cone(T)$. We claim that $\Gamma=\Gamma(t,T)=\Omega(|V(T^+)|,t)$ satisfies \ref{thm:tree+}, where $\Omega(\cdot,\cdot)$ comes from Theorem~\ref{thm:bigtwblurry}. Let $G\in \mathcal{E}_t$ be a graph of treewidth more than $\Gamma$. From Theorem~\ref{thm:bigtwblurry} and Observation~\ref{obs:blurry}, we deduce that there exists $X\subseteq V(G)$ such that $G[X]$ has a spanning subgraph isomorphic to the $2$-tree $T^+$. As a result, there exists a vertex $x\in X$ complete to $X\setminus \{x\}$ such that  $G[X\setminus \{x\}]$ has a spanning subgraph isomorphic to $T_f^f$. Since $G$ is $\{K_{3,3},K_t\}$-free, it follows from Theorem~\ref{KP} that $G[X\setminus \{x\}]$ contains $T_d^r$ (as an induced subgraph). Hence, $G[X\setminus \{x\}]$ contains $T$, and so $G$ contains $\cone(T)$, as required.
\end{proof}
\setcounter{section}{6}
\section{Acknowledgement}

Tara Abrishami worked with us on our initial attempts at Conjecture~\ref{diamondconj} using a different method (which led to \cite{twiv}), and Rose McCarty worked with us when we were suspicious of Theorem~\ref{mainchordal} and tried to find a counterexample. We thank them both.

\bibliographystyle{abbrv}
	\bibliography{ref}  

\end{document}